\documentclass{article}
\usepackage[utf8]{inputenc}
\usepackage{algorithm}
\usepackage{algorithmic}
\usepackage{amsmath}
\usepackage{amssymb}
\usepackage{amsthm}
\usepackage{tikz-cd}
\usepackage[textsize=scriptsize]{todonotes}
\usepackage{xcolor}
\usepackage{tikz}
\usepackage{enumitem}
\usetikzlibrary{cd, calc, shapes, decorations.shapes, decorations.markings, decorations.pathreplacing, shapes.geometric, shapes.symbols,positioning}
\usepackage{pgfplots}
\usepackage{float}
\usepackage{url}
\usepackage[normalem]{ulem} 
\usepackage{appendix}

\newtheorem{theorem}{Theorem}[section]
\newtheorem{prop}[theorem]{Proposition}
\newtheorem{proposition}[theorem]{Proposition}
\newtheorem{lemma}[theorem]{Lemma}

\theoremstyle{definition}
\newtheorem{definition}[theorem]{Definition}
\newtheorem{remark}[theorem]{Remark}
\newtheorem*{remark*}{Remark}

\newtheorem{example}[theorem]{Example}

\newcommand{\closure}[1]{\overline{#1}}
\newcommand{\lub}{\mathrm{l.u.b.}}
\newcommand{\ZZ}{\mathbb{Z}}
\newcommand{\RR}{\mathbb{R}}
\newcommand{\K}{\mathbb{K}}
\newcommand{\V}{\mathbb{V}}
\newcommand{\VV}{\mathcal{V}}

\newcommand{\KK}{\mathcal{K}}
\newcommand{\rank}{\rho}

\tikzset{twosimp/.style={fill opacity=0.6,fill=gray,draw opacity=0.9}}

\DeclareMathOperator{\vect}{Vect}
\DeclareMathOperator{\SC}{SC}
\DeclareMathOperator{\h}{H}
\DeclareMathOperator{\push}{push}

\title{Morse-based Fibering of the Persistence Rank Invariant}

\author{
  Asilata Bapat\footnote{Australian National University, Canberra, Australia}\\
  \texttt{asilata.bapat@anu.edu.au}
  \and Robyn Brooks \footnote{Boston College, Massachusetts, United States}\\
  \texttt{robyn.brooks@bc.edu} \\
  \and Celia  Hacker\footnote{EPFL, Lausanne, Switzerland} \\
  \texttt{celia.hacker@epfl.ch}
  \and Claudia Landi\footnote{ Università di Modena e Reggio Emilia, DISMI, Italy} \\ \texttt{claudia.landi@unimore.it} 
  \and Barbara I. Mahler \footnote{University of Oxford, United Kingdom}
  \\ \texttt{mahler@maths.ox.ac.uk} 
}

\begin{document}

\maketitle

\vspace{-0.3 cm}

\begin{abstract}
Although there is no doubt that multi-parameter persistent homology is a useful tool to analyse multi-variate data, efficient ways to compute these modules are still lacking in the available topological data analysis toolboxes. Other issues such as interpretation and visualization of the output remain difficult to solve. Software visualizing multi-parameter persistence diagrams is currently only available for $2$-dimensional persistence modules. One of the simplest invariants for a multi-parameter persistence module is its rank invariant, defined as the function that counts the number of linearly independent homology classes that live in the filtration through a given pair of values of the multi-parameter. We propose a step towards interpretation and visualization of the rank invariant for persistence modules for any given number of parameters. We show how discrete Morse theory may be used to compute the rank invariant, proving that it is completely determined by its values at points whose coordinates are critical with respect to a discrete Morse gradient vector field. These critical points partition the set of all lines of positive slope in the parameter space into equivalence classes, such that the rank invariant along lines in the same class are also equivalent. We show that we can deduce all persistence diagrams of the restrictions to the lines in a given class from the persistence diagram of the restriction to a representative 
in that class\footnote{
{\em MSC: 55N31, 57Q70}, {\em Keywords:} persistence module, persistence diagram, discrete gradient vector field, critical value
}.
\end{abstract}

\section*{Introduction}
Digital data are being produced at a constantly increasing pace, and their availability is changing the approach to science and technology. The fundamental hypothesis of Topological Data Analysis is that data come as samples taken from an underlying shape, and unveiling such shape is important for understanding the studied phenomenon.  Topological shape analysis amounts to determining non-trivial topological holes in any dimension. Computational Topology provides tools to derive specific signatures -- topological invariants -- which depend only on topological features of the shape of data and are robust to local noise \cite{Cohen-Steiner2007}. Among them, {\em persistent homology} \cite{Frosini92,Barannikov94,Robins2000,Edelsbrunner2002,Zomordian-Carlsson2005} stands out as most useful, having already found numerous applications in a diverse range of fields \cite{bendich2016persistent,bhattacharya2015persistent,green2019topology,lee2017quantifying,sinhuber2017persistent}. 
  
  The first step in the persistence pipeline is to build a family, called a filtration, of nested simplicial complexes that model the data at various scales by varying one or more parameters. The second step focuses on the maps induced in homology by the simplicial inclusions to extract invariants such as the {\em persistence module}. The third step is to use persistence invariants as a source of feature vectors in machine learning contexts, the final goal being to use the acquired topological information to improve the understanding of the underlying data. An important feature of this pipeline is its robustness to noise in the input data \cite{Cohen-Steiner2007}. 

 Some systems warrant analysis across multiple parameters, so it is important to focus on  multi-parameter persistence \cite{Carlsson2009} where the filtration may depend on any number of real-valued parameters, and $n$-tuples of parameter values have an inherited partial order. Unlike single-parameter persistent homology,  which is completely described by a {\em persistence diagram}, multi-parameter persistence modules contain more information than it is possible to handle, understand, and visualize easily. Therefore, it is convenient to summarize them by simpler invariants.  
 
Among the various invariants considered for a multi-parameter persistence module, such as the blockcodes of \cite{dey2020generalized} and the multi-graded Betti numbers of \cite{lesnick2020computing}, one of the simplest invariants  is the {\em rank invariant} \cite{Carlsson2009}, defined as the function that counts the number of linearly independent homology classes that live in the filtration through any given pair of values of the multi-parameter. Theoretically, computation of the rank invariant can be carried out by fibering it along lines with positive slope in an $n$-dimensional space, where $n$ is the number of parameters \cite{Cerri-et-al2013}. Indeed, the  rank invariant of the restriction 
  of a persistence module to an increasing line  is completely described by a persistence diagram.
  The union of all such persistence diagrams forms a compact object called the persistence space \cite{CerriLandi}. Practically though, one needs to restrict the  number of relevant lines to a reasonable number. 
   For example, the state-of-the-art tool for rank invariant visualization RIVET \cite{Lesnick-Wright2015} achieves reduction to some template lines using Betti tables, takes $O( m^3)$ runtime (with $m$ the number of simplices), and is limited to two parameters.

 Given the dimensional, computational, and interpretability limits of the currently available methods,   multi-parameter persistence is not yet a viable option for data analysis.  It is thus worth  exploring different approaches that could enable faster computation and  better understanding of the rank invariant  in a dimension-agnostic way. To this end, we propose to exploit the information contained in the critical cells of a discrete gradient vector field consistent with the given multi-filtration as a means to enhancing geometrical understanding and computational efficiency of multi-parameter persistence.

Evidence for the usefulness of Morse theory in multi-parameter persistence is given in \cite{ReducingComplexes,Allili2019,Scaramuccia-et-al-2020}. In these papers, discrete Morse theory is used to reduce the multi-parameter persistence input data size by substituting the original simplicial complex with a Morse complex containing only critical cells but having the same persistence as the initial complex. This approach takes advantage of the fact that number of critical cells is very small in comparison with $m$, the total number of simplices. Computations are sped up by reducing the actual value of $m$, while the runtime complexity of obtaining a discrete gradient vector field compatible with the filtration is $m\cdot s^2$, where $s$ denotes the maximum number of simplices in a vertex star. Although in the worst case $s$ may be as large as $m$, in many applications it is negligible in comparison to $m$, so that the complexity may be considered linear in $m$. Moreover, as shown by tests in \cite{Scaramuccia-et-al-2020}, this global reduction also avoids repeating the retrieval of the same null persistence pairs when fibering the rank invariant along different lines, improving time performances proportionally to the number of lines: the larger the number of lines, the more convenient the reduction to the Morse complex.

Supported by such empirical evidence, in this paper we investigate the theoretical connection between critical cells of discrete Morse theory and rank invariant computation along lines, with the goal of improving the available methods for fibering the rank invariant of multi-parameter persistence in a way that is computationally efficient,  geometrically interpretable, and readily visualizable.  To this end, we exploit the correspondence between critical cells of a discrete gradient vector field and topological changes in filtrations compatible with it to:
  \begin{enumerate}[noitemsep]
    \item Show that the rank invariant for $n$-parameter persistence modules can be computed by selecting a small number of values in the parameter space $\RR^n$ determined by the critical cells of the discrete gradient vector field, and using these values to reconstruct the rank invariant for all other possible values in the parameter space (Theorem \ref{thm:rankinv}). 
    \item Use such critical values to define an equivalence relation among lines so that representatives from each equivalence class form a convenient selection of lines for fibering the rank invariant (Definition \ref{def:rec-position}).
   \item Show that fibering the rank invariant along lines in the same equivalence class yields persistence diagrams that can be obtained one from the other by simple bijections between critical values (Theorem \ref{thm:representative-dgm}). 
   \item Present a method based on the previous results for computing any fiber of the rank invariant using only finitely many template fibers (Section \ref{sec:conclusion}).
  \end{enumerate}
We emphasize that all of our results hold for any number of parameters, thus improving the state-of-the-art methods that allow for only two. Moreover, our method's input requires only the critical cells of a discrete gradient vector field, so that the required pre-processing can be achieved in just linear time according to various available algorithms \cite{ReducingComplexes,Allili2019,Scaramuccia-et-al-2020}. Finally, the connection to Morse theory allows for a more immediate geometric interpretation and visualization of the topological features detected by persistence as pairs of critical cells in the given simplicial complex.\\

The outline of this paper is as follows:
in Section \ref{sec:Notation and Defs}, we outline background information which will be necessary for the reader to understand the rest of the text. Section \ref{sec:rank_inv} focuses on computing the rank invariant, supported by lemmas and diagram-chasing. Section \ref{sec:computing_pers_space} explains the computation of the persistence space along with delving into fibering the rank invariant along equivalent lines. We conclude with Section \ref{sec:conclusion} by laying out a method for fibering the rank invariant of persistence by lines, and suggesting possible applications. Finally, in Appendix \ref{app:cutsandpairs} we discuss a result that allows us to develop an algorithm which chooses a representative line for each equivalence class in the case of $2$-parameter persistence modules. The generalization of this algorithm to higher dimensional persistence modules remains future work.

\section{Notation and Definitions}
\label{sec:Notation and Defs}
These definitions are partially based on the definitions in \cite{Lesnick-Wright2015,Kerber-Lesnick-Oudot2018} and \cite{forman2001}.

Let $\K$ be a  field. For computational purposes, $\K$ is often taken to be finite.  Define the following partial order on $\RR^n$: for $u=(u_i),v=(v_i)\in\RR^n$,  we say that $u\preceq v$ (resp.  $u\prec v$) if and only if $u_i\leq v_i$ (resp. $u_i<v_i$) for all $i$. The poset $(\RR^n,\preceq)$ will be our {\em parameter space}.

\begin{definition}[Persistence module]\label{def:PM}
  A {\em persistence module} $\V$ over the parameter space $\RR^n$ is an assignment of a $\K$--vector space $V_u$ to each $u\in\RR^n$, and transition maps $i^{u,v}:V_u\to V_v$ to pairs of points $u\preceq v\in\RR^n$, satisfying the following properties:
  \begin{itemize}
  \item $i^{u,u}$ is the identity map for all $u \in\RR^n$.
  \item $i^{v,w}\circ i^{u,v}=i^{u,w}$ for all $u\preceq v\preceq w \in\RR^n$.
  \end{itemize}
  A persistence module over $\RR^n$ is also known as an $n$-parameter persistence module or an $n$D persistence module.
\end{definition}

\begin{definition}[Rank invariant]\label{def:rank-inv}
For $n\geq 1$, let $\mathcal{H}^n\subset\RR^n\times\RR^n$ be the subset of pairs $(u,v)$ such that $u\preceq v$.  Let $\V$ be an $n$D persistence module. Then the {\em rank invariant of $\V$} is a function $\rank_{\V} \colon \mathcal{H}^n \to \ZZ$, defined as
\[\rank_{\V}(u,v) = \textrm{rank}(i^{u,v}).\]
\end{definition}

Persistence modules most often arise from filtered simplicial complexes.

\begin{definition}[Finite simplicial complex]\label{def:finite-simp-comp}
A {\em finite simplical complex} $K$ is a collection of subsets of a finite set  $V_0$, such that:
\begin{itemize}
\item all singletons of $V_0$ are in $K$, and
\item If $\beta\in K$ and $\alpha\subset\beta$, then $\alpha\in K$.
\end{itemize}
The elements of $V_0$ are called {\em vertices}, and the elements of $K$ are called {\em simplices}.  If $\alpha\in K$ contains $p+1$ vertices, then we say $\alpha$ has {\em dimension} $p$, sometimes denoted $\alpha^{(p)}$.  If $\beta\in K$ and $\alpha \subset \beta$, then we say $\alpha$ is a {\em face} of $\beta$ and $\beta$ is a {\em coface} of $\alpha$, and denote this by $\alpha<\beta$. If $\alpha$ is a codimension one face of $\beta$, we say that $\alpha$ is a {\em facet} of $\beta$ and $\beta$ is a {\em cofacet} of $\alpha$.  If $\alpha$ is a face of a dimension $p$ simplex $\beta$, and is not a face of any other $p$ dimensional simplex, then we say that $\alpha$ is a {\em free face} of $\beta$.
\end{definition}


\begin{definition}[Filtration]\label{def:filtration}
  Let $K$ be a finite simplicial complex. An \emph{$n$-parameter filtration} of $K$ is a collection of subcomplexes $\KK=\{K^u\}_{u\in\RR^n}$ of $K$, such that $K^u \subset K^v$ whenever $u \preceq v$, and moreover that
  \[
    K = \bigcup_{u\in \RR^n}K^u.
  \]
  A complex with such a filtration is called a \emph{multi-filtered simplicial complex} or an \emph{n-filtered simplicial complex}.
\end{definition}
\begin{remark*}
Since $K$ is finite, there is some $u \in \RR^n$ such that $K = K^u$.
\end{remark*}

\begin{definition}\label{def:entrance-value}
  We say that a simplex $\sigma$ in $K$ has {\em entrance value} $u\in\RR^n$ if
  \[\sigma\in K^u-\bigcup_{w\preceq u, u \neq w}K^w\]
 \end{definition}
 
 This value tells us where $\sigma$ has entered the filtration. Note that if the entrance value of $\sigma$ is $u$, then $\sigma\in K^v$ for all $u\preceq v$, and $\sigma \notin K^w$ for all $w \preceq u$ such that $u \neq w$.
 Therefore, if $u$ is an entrance value of $\sigma$, then $u$ is a (possibly non-unique) minimal value of the set $\{v \in \RR^n \mid \sigma \in K^v\}$.
 
 According to the definition of filtration, if a simplex $\sigma$ enters the filtration at $u$, all of its faces have to be in $K^u$ as well. The faces can enter jointly with $\sigma$ or at  earlier values of the filtration. However, entrance values are not guaranteed to exist.  To avoid pathological situations in which a simplex does not have any entrance values, such as in the 1-parameter  filtration of $K$ defined by $K^u=\emptyset$ for each $u\le 0$ and $K^u=K$ for each $u>0$, we will only consider filtrations as follows.

\begin{definition}\label{tamenes}(Tameness)  A filtration of a simplicial complex $K$ is said to be {\em tame} if each simplex of $K$ has at least one entrance value. 
\end{definition}

\begin{remark*}
For a tame filtration, it is guaranteed that there exists some $u \in \RR^n$ such that $K^u=\varnothing$.
\end{remark*}

\begin{example}
  Consider a function $f\colon K \to \RR^n$ that is monotonic with respect to the face relation.
  That is, $\alpha < \beta \in K$ implies that $f(\alpha)\preceq f(\beta)$.
  Every such function gives rise to a tame filtration, by defining the filtered pieces of $K$ to be sublevel sets as follows:
  \[K^v = f^{-1}( \{u\in\RR^n:u \preceq v\}).\]
\end{example}
On the other hand, not all filtrations of $K$ as in Definition \ref{def:filtration} define a function $f: K \to \RR^n$, even under the tameness assumptions, because the entrance value of a simplex may not be unique. This motivates the following definition that can be found in \cite{Carlsson}.

\begin{definition}[One-criticality]\label{def:one-criticality}
An $n$-parameter filtration is said to be \emph{one-critical} if every simplex of $\KK$ has a unique entrance value in the filtration. 
\end{definition}

\begin{prop}
  Each tame one-critical filtration of a simplicial complex $K$ is the sublevel set filtration of a monotonic function $f\colon K\to \RR^n$.
  Conversely, if $f \colon K \to \RR^n$ is monotonic, then its sublevel sets form a tame and one-critical filtration of $K$.
\end{prop}

\begin{proof}
  Let $K$ be a simplicial complex and $\KK$ a tame and one-critical filtration of $K$.
  We construct a monotonic function $f_{\KK} \colon K \to \RR$ whose sublevel set filtration is precisely $\KK$.
  For each simplex $\sigma \in \KK$, set $f_{\KK}(\sigma)$ to be the unique entrance value of $\sigma$ in $K$ with respect to $\KK$, which exists and is unique because the filtration is tame and one-critical.
  Let $\alpha, \beta \in K$ such that $\alpha < \beta$.
  Recall that if $\beta \in K^u$ for some $u\in \RR^n$, then $\alpha \in K^u$ as well.
  In particular, $\alpha \in K^{f_{\KK}(\beta)}$.
  Since $f_{\KK}(\alpha)$ is the unique minimum value $u$ such that $\alpha \in K^u$, it must be the case that $f_{\KK}(\alpha) \preceq f_{\KK}(\beta)$, and therefore $f_{\KK}$ is monotonic.
  
  Conversely, let $f \colon K \to \RR^n$ be a monotonic function.
  As before, set $\KK$ to be the set of all sublevel sets 
  \[ K^v = f^{-1}(\{u \in \RR^n \colon u \preceq v\}).\]

  Let $\sigma$ be any simplex of $K$.
  It is clear that $\sigma \in K^{f(\sigma)}$.
  We prove that $f(\sigma)$ is the unique entrance value of $\sigma$.
  Note that it is indeed an entrance value: if $u \preceq f(\sigma)$ and $u \neq f(\sigma)$, then $f(\sigma) \notin K^u$.
  Suppose that $v$ is any entrance value of $\sigma$, which means that $\sigma \in K^v$ and $\sigma \notin K^u$ for any $u \neq v$ such that $u \preceq v$.
  Since $\sigma \in K^v$, we must have $f(\sigma) \preceq v$.
  Since $\sigma \notin K^u$ for any $u \neq v$ such that $u \preceq v$, it must be the case that $f(\sigma) \neq u$ for any such $u$.
  Therefore $v = f(\sigma)$.
  Since every simplex $\sigma \in K$ has the unique entrance value $f(\sigma)$, we obtain a tame and one-critical filtration of $K$ from the sublevel sets of $f$.
\end{proof}

To avoid pathologies, we always assume the filtrations considered in this paper are tame and one-critical.

Now think of the poset $\RR^n_\preceq=(\RR^n,\preceq)$ as a category where the objects are elements of $\RR^n$, and the morphisms are given by the order relation.
Let $\SC$ be the category of finite simplicial complexes with inclusions as the morphisms.
A filtered simplicial complex can be thought of as a functor $\KK: \RR^n_\preceq \longrightarrow \SC$.
Further, for each $i \in \ZZ$, we can take the $i$-th homology of a simplicial complex to obtain a vector space.
The two functors above can be composed to obtain a functor from $\RR^n_\preceq$ to $\vect_{\K}$, the category of $\K$-vector spaces.
By Definition~\ref{def:PM}, a persistence module can be viewed as a functor from $\RR^n_\preceq$ to $\vect_{\K}$, so the construction above is a special case of a persistence module.
This motivates the following definition.

\begin{definition}[Persistent homology]\label{def:persistence-hom}
For an integer $i$, the $i$-th {\em multi-parameter persistent homology group} is the persistence module 
\[ H_i\KK: \RR^n_\preceq \longrightarrow \vect_{\K}
\]
defined as the composition of the filtration functor $\KK$ with the $i$-th homology functor for simplicial complexes:
\begin{center}
\begin{tikzcd}
\RR^n_\preceq  \ar[r,"{\KK}"] \ar[drru ,"{ H_i\KK }"' , bend right=20] & \SC \ar[r, "  {\h_i}  "]  & \vect_{\K}. 
\end{tikzcd}
\end{center}
\end{definition}

We now introduce the basic constructions of discrete Morse theory, which we will heavily use in the remainder of the paper.
More details may be found in~\cite{forman2001}.

%
%

\begin{definition}[Discrete vector field]\label{def:dvt}
  A {\em discrete vector field} $\VV$ on $K$ is a collection of pairs of simplices $(\alpha,\beta)$ of $K$ where $\alpha$ is a facet of $\beta$, such that each simplex is in at most one pair of $\VV$.
  Given a discrete vector field $\VV$ on a simplicial complex $K$, a {\em $V$-path} is a sequence of simplices of dimensions $p$ and $p+1$,
  \[\alpha_0,\beta_0,\alpha_1,\beta_1,\alpha_2,\dots,\beta_r,\alpha_{r+1}
  \]
 such that for each $i=0,\dots,r$, we have $(\alpha_i,\beta_i)\in \VV$ and $\beta_{i}>\alpha_{i+1}\neq\alpha_i$.
  A path is called a {\em non-trivial closed path} if $r>0$ and $\alpha_0=\alpha_{r+1}$.
\end{definition}


\begin{definition}\label{def:gradient-vf}
A discrete vector field $\mathcal{V}$ on a simplicial complex $K$ is called a \emph{gradient vector field} if it contains no non-trivial closed $V$-paths. A simplex $\sigma\in K$ is \emph{critical} if it is not paired in $\VV$.
\end{definition}
An example of such a discrete gradient vector field can be found in Figure \ref{fig:dgvf}. 

\begin{figure}[H]
\begin{center}
\begin{tikzpicture}
\input{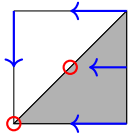}
\end{tikzpicture}
\caption{An example of discrete gradient vector field. The pairs in the vector field are denoted by arrows and the critical cells are marked by a circle.} \label{fig:dgvf}
\end{center}
\end{figure}

\begin{definition}[Elementary collapse]\label{def:elementary-collapse}
 Let $K_1, K_2$ be simplicial complexes such that $K_2\subset K_1$, and $K_1 \setminus K_2=\{\alpha,\beta\}$ where $\alpha$ is a free facet of $\beta$.  Then the combinatorial deformation retract of $K_1$ to $K_2$ given by removing $\alpha$ and $\beta$ is called an {\em elementary (simplicial) collapse}.  The pair $(\alpha,\beta)$ is called a \emph{collapsing pair}.
\end{definition}
\begin{remark*} The definition of elementary collapse is also valid for CW-complexes.
\end{remark*}


\begin{example}
In Figure \ref{fig:collapse} we can see how to collapse the complex in Figure \ref{fig:dgvf} using the given discrete gradient vector field. The first collapsing pair is formed by the edge $\alpha$ and the triangle $\beta$. The second collapse uses the pair $(\gamma, \delta)$, given by the edge $\delta$ and the vertex $\gamma$.


\begin{figure}[H]
\begin{center}
\begin{tikzpicture}
\input{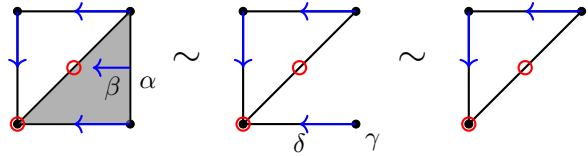}
\end{tikzpicture}
\caption{The complex on the right is obtained by the one on the left by first collapsing the pair $(\alpha, \beta)$, then the pair $(\gamma, \delta)$.} \label{fig:collapse}
\end{center}
\end{figure}
\end{example}

An elementary simplical collapse is a homotopy equivalence between two simplicial complexes, which in turn induces an isomorphism on the level of homology.  Therefore, if $K_1$ and $K_2$ are related through a series of elementary collapses, then $H_i(K_2)\cong H_i(K_1)$ for all $i$. 

In the case when a simplicial complex has no free faces available, it is still possible to simulate an elementary collapse, which in this case is called {\em internal}, by first removing a critical cell in order to obtain a free face, and then reinserting it after updating the incidence relations. The cell complex obtained in this case may no longer be simplicial, but internal  collapses still induce isomorphisms in homology.

To turn simplicial collapses from just homology preserving into persistent homology preserving transformations, it is convenient to confine ourselves to considering gradient vector fields compatible with filtrations.

\begin{definition}[Consistency]\label{def:consistency}
  A discrete gradient vector field $\VV$ on a simplicial complex $K$  is said to be {\em consistent} (or \emph{compatible}) with a  filtration $\KK=\{K^u\}_{u\in\RR^n}$ of  $K$  if the following condition is satisfied:
$$\forall(\sigma,\tau)\in \VV, \sigma\in K^u \iff \tau \in K^u.$$ 
In this case, we say that a value $u\in\RR^n$ is a {\em critical value of $K$} if it is the entrance value of a critical cell of $\VV$.  
\end{definition}

\begin{example}\label{ex:bifiltration}
In Figure \ref{fig:filtration} we see a bifiltration, i.e. a filtration with $2$ parameters, of a finite simplicial complex. This filtration is one-critical and the gradient vector field is consistent with the given filtration.

\begin{figure}[H]
\begin{center}
\begin{tikzpicture}
\input{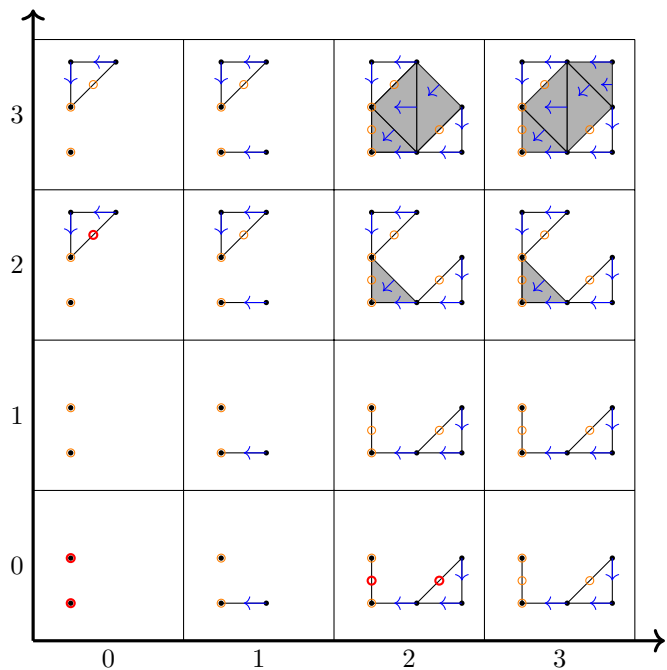}
\end{tikzpicture}
\caption{Each box corresponds to a multi-parameter in $\RR^2$, showing the simplices that are present at the corresponding step of the filtration. As before, the pairs of the discrete gradient vector field are indicated by arrows and the critical cells by circles. As we will see later on, we are interested in the entrance values of the critical cells. When the critical cells enter the filtration, they are denoted in a darker red, whereas they are denoted in a lighter orange for larger values, so as to not lose track of them throughout the filtration.} \label{fig:filtration}
\end{center}
\end{figure}

\end{example}

\section{Computing the Rank Invariant}
\label{sec:rank_inv}
Let $K$ be a finite simplical complex, $\VV$ be a discrete gradient vector field on $K$, and $\KK=\{K^u\}_{u\in\RR^n}$ be a one-critical $n$-parameter filtration on $K$ which is consistent with $\VV$. Then, for each integer $i$, there is an $n$-parameter persistence module $\V_i$ where the $\K$-vector space associated to $u\in\RR^n$ is $H_i(K^u)$, the $i$-th homology group of $K^u$. Furthermore, for $u\preceq v\in\RR^n$, $i_{u,v}$ is the induced map, on the level of homology, of the inclusion of $K^u$ into $K^v$.  The goal of this section is to identify a finite subset of values in $\RR^n$ from which the rank invariant of $\V_i$ can be computed.  In other words, we want some finite $U\subset\RR^n$ such that, for all $u\preceq v\in\RR^n$, there exists $\overline{u}\preceq \overline{v}\in U$ such that $\rho_{\V_i}(u,v)=\rho_{\V_i}(\overline{u},\overline{v})$.

Define $C$ to be the set of critical values of  $\VV$:
\[
C=\{u\in\RR^n|\sigma\textrm{ is critical in }\VV,\textrm{ the entrance value of }\sigma\textrm{ is }u\}.
\]
Let $\overline{C}$ be the closure of $C$ under least upper bound.  Theorem \ref{thm:rankinv} states that our candidate set $U$ as described above is exactly given by $\overline{C}$.  In order to identify each $u\in\RR^n$ with an element of $\overline{C}$, define $\closure{u} = \max \{u' \in \closure{C} | u'\preceq u\}$.

\begin{example} Figure \ref{fig:u_bar} shows the set of critical values and its closure in the parameter space for the bifiltration in Example \ref{ex:bifiltration}. We can see there that for the value $u$, its corresponding $\bar{u}$ is the critical value $c_2$.

\begin{figure}[H]
\begin{center}
\begin{tikzpicture}
\input{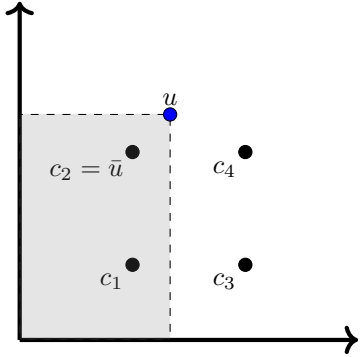}
\end{tikzpicture}
\caption{ Taking the set of critical values $C$ to consist of $\{c_1 , c_2 ,c_3\}$, $\overline{C} = C \cup\left\lbrace c_4 \right\rbrace $.   The grey area contains all values $x \preceq u$, and in this case $\bar{u} = c_2$.} \label{fig:u_bar}
\end{center}
\end{figure}
\end{example}

\begin{lemma}\label{lem:max}
Let $U$ be a non-empty finite set closed under least upper bound. For all $u$ in $\RR^n$, the set $U'=\{u' \in U| u'\preceq u\}$, if non-empty, admits a (unique) maximum.  
 \end{lemma}

\begin{proof}
From the fact that $U$ is non-empty and finite, it admits maximal elements. Let $a$ and $b$ be maximal elements in $U'$. Because  $U$ is closed under least upper bound, there is $c$ in $U$ such that $a,b\preceq c \preceq u$. Hence, $c\in U'$. Thus, $c=a=b$ by maximality of $a$ and $b$ in $U'$. 
\end{proof}

The idea behind the definition of $\overline{C}$ is as follows.  Since the filtration $\KK$ has a consistent discrete gradient vector field $\VV$, the tools of discrete Morse theory may be used to identify which elements of $\KK$ are guaranteed to be homotopy equivalent, and therefore have isomorphic homology groups.  Lemma \ref{lem:reduction}    
can be used to show the existence of a simplicial collapse, induced by $\VV$, between certain elements of $\KK$; this simplicial collapse is a homotopy equivalence.

\begin{lemma}\label{lem:reduction}
Let $\KK=\{K^u\}_{u\in\RR^n}$  be a one-critical filtration and $\VV$ a discrete gradient vector field on a finite simplicial complex $K$ consistent with $\KK$.  For $u\in\RR^n$, let $\closure{u} = \max \{u' \in \closure{C} | u'\preceq u\}$, with $C$ the set of critical values of $\VV$ and $ \closure{C} $ its closure under least upper bound. If $K^u-K^{ \closure{u}}$ is non-empty, then it contains two simplices $\sigma,\tau$ such that  $(\sigma,\tau)\in \VV$  and $\sigma$ is a free facet of $\tau$.  
\end{lemma}

\begin{proof}
Because $K^u-K^{ \closure{u}}$ is non-empty and finite, we can take $\tau\in K^u-K^{ \closure{u}}$ of maximal dimension. Denote by $\lub(u,v)$ the least upper bound of two elements $u,v\in \RR^n$. If $\tau$ is critical,  there is a critical value $\hat u$ such that  $\closure{u}\ne  \lub(\hat u, \closure{u})\preceq u$. Since $\lub(\hat u, \closure{u})\in \closure{C}$, this contradicts the definition of $\closure{u}$.

 Thus, it must be that $\tau$ is non-critical, and belongs to a vector $(\sigma,\tau)\in \VV$. Because $\VV$ is consistent with $\KK$, $\sigma$ must belong to  $K^u-K^{ \closure{u}}$ as well, and since $\tau$ is of maximal dimension, it must be that $\sigma$ is a face of $\tau$.  Hence, each of the simplices of maximal dimension in $K^u-K^{ \closure{u}}$ belongs to a vector $(\sigma,\tau)\in \VV$ with $\sigma$ also in $K^u-K^{ \closure{u}}$, and there are finitely many of them because $K$ is finite.
 
 We claim that one of the pairs of simplices $(\sigma,\tau)\in \VV$, with $\tau$ of maximal dimension in $K^u-K^{ \closure{u}}$, must be a collapsing pair, meaning $\sigma$ is a free facet of $\tau$.  Supposing that this is not the case, we create a non-trivial cyclic path in $\VV$, contradicting the assumption that $\VV$ is a discrete gradient vector field.
 
 Choose a pair $(\sigma_1,\tau_1)\in \VV$  of simplices in $K^u-K^{ \closure{u}}$, with $\tau_1$ of maximal dimension. If $\sigma_1$ is not a free facet of $\tau_1$, there exists another cofacet $\tau_2\in K^u$ such that $\sigma_1<\tau_2\neq\tau_1$.  As $\VV$ is consistent with $\KK$, $\sigma_1$ cannot be added to the filtration after $\tau_2$, so it must be that $\tau_2\in K^u-K^{ \closure{u}}$ as well.  Moreover, since the dimension of $\tau_1$ is  maximal and $\sigma_1$ is a facet of $\tau_2$,  $\tau_1$ and $\tau_2$ must have the same dimension. Thus also $\tau_2$ is of maximal dimension.  Thus, $\tau_2$ is in some non-critical pair $(\sigma_2,\tau_2)\in \VV$  of simplices in $K^u-K^{ \closure{u}}$.  Finally, since $\VV$ is a discrete gradient vector field, $\sigma_1$ can only exist in one pair of $\VV$, so that $\sigma_2\ne \sigma_1$.
 
 We may iterate the above argument.  After an appropriate finite number of iterations, we obtain the following $V$-path:
 
 \[
 \sigma_1<\tau_1>\sigma_2<\tau_2>\dots<\tau_{n-1}>\sigma_n<\tau_n>\sigma_{n+1}.
 \]
 
Since there are only finitely many possible choices for $\sigma_i$, it must be that $\sigma_i=\sigma_j$ for some $i\neq j\in\{1,\dots,n+1\}$.  The portion of the $V$-path between $\sigma_i$ and $\sigma_j$ is non-trivial and cyclic. This is a contradiction, as $\VV$ is a discrete gradient vector field. Thus,  for some pair $(\sigma,\tau)\in \VV$ of simplices in $K^u-K^{ \closure{u}}$ with $\tau$ of maximal dimension, $\sigma$ must be  a free facet of $\tau$, proving our claim.
  \end{proof}

We may now prove the main result of the section, which gives a formula for computing the rank invariant for any pair $(u,v)$, using elements of $\overline{C}$.  

\begin{theorem}\label{thm:rankinv}
   Let $K$ be a finite simplicial complex, and let $\KK=\{K^u\}_{u\in\RR^n}$ be a filtration on $K$.  Suppose $\VV$ is a discrete gradient vector field on $K$  consistent with $\KK$. Let $\V_i=H_i(\KK)$. Then, for all $u\preceq v$, 
   \[\rank_{\V_i}(u,v)=\rank _{\V_i} (\closure{u},\closure{v})\]
   with 
   \[\closure{u} = \max \{u' \in \closure{C} | u'\preceq u\},\]
   \[\closure{v} = \max \{v' \in \closure{C} | v'\preceq v\}.\] 
if $\{u' \in \closure{C} | u'\preceq u\}$ is non-empty, and $\rank_{\V_i}(u,v)=0$ otherwise.
 \end{theorem}

\begin{proof}
Let $u,v\in\RR^n$. Because $C$ is non-empty and $\overline C$ contains all least upper bounds, if $\{u' \in \closure{C} | u'\preceq u\}=\emptyset$, then $K^u=\emptyset$, yielding $\rank_{\V_i}(u,v)=0$. Otherwise, if $\{u' \in \closure{C} | u'\preceq u\}$ is non-empty, then also $\{v' \in \closure{C} | v'\preceq v\}$ is so, and  $\closure{u},\closure{v}$ exist and are unique by Lemma \ref{lem:max}.

Set $K_0:=K^u$. If $K_0=K^{\closure{u}}$, then obviously $H_i(K^u)= H_i(K^{\closure{u}})$ for all $i$. Otherwise, $K_0-K^{\closure{u}}$ is non-empty and by Lemma \ref{lem:reduction} there is a pair  $(\sigma_1,\tau_1)\in \VV$  of simplices of $K_0-K^{\closure{u}}$ such that $\sigma_1$ is a free facet of $\tau_1$. Let $K_1$ be the simplical complex obtained by performing the elementary collapse of $\sigma_1$ onto $\tau_1$; note that this means that the map induced on homology by the inclusion of $K_1$ into $K_0$ is an isomorphism: $H_i(K_1)\cong H_i(K_0)$.  Now, either $K_1=K^{\closure{u}}$, in which case  obviously $H_i(K_1)= H_i(K^{\closure{u}})$ for all $i$, or, we may restrict $\VV$ to $K^1$ and consider the filtration of $K_1$ induced from that of $K_0$. By applying Lemma \ref{lem:reduction} to $K_1$, there is  a vector in $\VV$ whose elementary collapse gives $K_2$ such that 
 the map induced by the inclusion of $K_2$ into $K_1$ is an isomorphism: $H_i(K_2)\cong H_i(K_1)$. By induction, for any $r\ge 1$, either $K_r=K^{\closure{u}}$, in which case  obviously $H_i(K_r)= H_i(K^{\closure{u}})$ for all $i$, or, we may restrict $\VV$ to $K_r$ and consider the filtration of $K_r$ induced from that of $K_{r-1}$. By applying Lemma \ref{lem:reduction} to $K_r$, there is  a vector in $\VV$ whose elementary collapse gives $K_{r+1}$ such that  the map induced by the inclusion of $K_{r+1}$ into $K_1$ is an isomorphism: $H_i(K_{r+1})\cong H_i(K_r)$. By the finiteness of $K$, there must be a value of $r$ such that $K_r=K^{\closure{u}}$, yielding that the map induced by the inclusion of $K^{\closure{u}}$ into $K^u$ is an isomorphism for all $i$:
\[
H_i(K^u)\cong H_i(K^{\overline{u}})
\]
Analogous argument works for $v$:
\[
H_i(K^v)\cong H_i(K^{\overline{v}}).
\]
Because the above isomorphisms are induced by inclusions,  the following diagram commutes and gives equality of the $i^{th}$ rank invariants of the pairs $(u,v)$ and $(\overline{u},\overline{v})$:
\[
\begin{tikzcd}
H_i(K^u) \arrow[r,"i_*^{u,v}"] \arrow[d, "\cong"]& H_i(K^v) \arrow[d, "\cong"] \\
H_i(K^{\overline{u}}) \arrow[r,"i_*^{\overline{u},\overline{v}}"] &H_i(K^{\overline{v}}).
 \end{tikzcd}
\]
\end{proof}

\section{Computing the persistence space}
\label{sec:computing_pers_space}
For the sake of visualization, the rank invariant of an $n$-parameter persistence module $\V$ can be completely encoded as a multiset of points known as a {\em persistence diagram} when $n=1$ \cite{Cohen-Steiner2007}, and as a {\em persistence space} when $n\ge 1$ \cite{CerriLandi}. By completeness of the encoding we mean that the rank invariant can be exactly reconstructed from the persistence space   (cf. the $k$-Triangle Lemma in \cite{Cohen-Steiner2007} and the Representation Theorem  in \cite{CerriLandi}). 

While the persistence space is easier to visualize than the rank invariant, as it is a set of points rather than a function, still for an $n$-parameter persistence module it lives in a $2n$-dimensional space. So, it is convenient  to visualize it along fibers \cite{Cerri-et-al2013}. For example, PHOG \cite{PHOG} and RIVET \cite{Lesnick-Wright2015} visualize the persistence space of a $2$-parameter persistence module by fibering it through lines. 

The goal of this section is to propose a computational procedure to recover such fibration along lines for persistence modules with any number of parameters, by using critical values of gradient vector fields. 

We start by reviewing the necessary definitions and properties.\\

A point $(u, v) \in  \mathcal{H}^n$ belongs to the persistence space $\mathrm{spc}({\V})$ if and only if its {\em multiplicity}
\begin{eqnarray}
\mu_{\V}(u,v) := \min_{\scriptsize{\begin{array}{c}
\vec e\succ  0\\
u+\vec e\prec v-\vec e\end{array}}}\rank_{\V}(u+\vec e,v-\vec e)- \rank_{\V}(u-\vec e,v-\vec e)+\\ -\rank_{\V}(u+\vec e,v+\vec e)+\rank_{\V}(u-\vec e,v+\vec e)
\end{eqnarray}\label{def:mu}
is positive. This corresponds to the number of independent cycles that, along a positive direction in the parameter space,  appear  at $u$ and become boudaries at $v$.  

Similarly, a point $(u, \infty)$ belongs to the persistence space of $\V$ if and only if its multiplicity
\begin{eqnarray}
\mu_{\V}(u,\infty) := \min_{\scriptsize{\begin{array}{c}
\vec e\succ  0\\
 v\succeq u\end{array}}}\rank_{\V}(u+\vec e,v)- \rank_{\V}(u-\vec e,v)
\end{eqnarray}\label{def:mu-infty}
is positive. This corresponds to the number of  independent cycles that, along a positive direction in the parameter space, appear exactly at $u$ and persist for every larger value of the parameter. 

In both cases, the multiplicity can be computed by fixing a direction for $\vec e$ and only varying its length (with alternate sums of the ranks decreasing as the length decreases). Two convenient directions for $\vec e$ are the diagonal direction and the $v-u$ direction. Moreover, for points at infinity, the multiplicity is reached for increasing values of $v$. 

In particular, for $n=1$, the persistence space is the persistence diagram of a $1$-parameter persistence module. In terms of intervals in a persistence module bar decomposition, points in $\mathcal{H}^n$ of positive multiplicity correspond to finite intervals, points at infinity of positive multiplicity correspond to infinite intervals.

\subsection{Restriction of a persistence module to lines}

 Given a line $L$ contained in the parameter space $\RR^n$, each point $u\in L$ can be written as $u=\vec mt+u_0$, with $u_0$ a fixed starting point on $L$, $\vec m\in\RR^n$ a fixed velocity vector,  and $t$ a real parameter.  If $\vec 0\prec \vec m$, we say that $L$ has {\em positive slope}.

For  an $n$-parameter persistence module $\V$ and a line $L\subseteq \RR^n$ with positive slope, the {\em restriction} of $\V$ to $L$  is the persistence module $\V_L$ that assigns  $V_u$ to each $u\in L$, and whose transition maps $i_L^{u,v}:V_u\to V_v$ for $u\preceq v\in L$ are the same as in $\V$. Once a parametrization $u=\vec mt+u_0$ of $L$ is fixed, the persistence module $\V_L$ is isomorphic to the 1-parameter persistence module, by abuse of notation denoted by $\V_L$, that assigns to each $t\in\RR$ the vector space $(\V_L)_t=\V_u$, and to $s<t\in\RR$, the transitions maps $i^{s,t}=i_L^{u,v}=i^{u,v}$. 

By construction, for  $u=\vec ms+u_0$ and $v=\vec mt+u_0$, it holds that
$$\rank_\V(u,v)=\rank_{\V_L}(s,t).$$
Hence, the multiplicity of a point $(u, v) \in  \mathcal{H}^n$ in $\mathrm{spc}({\V})$ coincides with that of $(s,t)\in \mathcal{H}^1$ in $\mathrm{dgm}({\V_L})$:  
$$\mu_{\V}(u,v)=\mu_{\V_L}(s,t).$$

In conclusion, the persistence space $\mathrm{spc}({\V})$ can  be viewed as the fibered union of infinitely many persistence diagrams $\mathrm{dgm}({\V_L})$, each associated with a line $L$ with positive slope.

\subsection{Critical values determine the persistence space}

Our next goal is to demonstrate that, for a persistence module $\V$ obtained from a tame and one-critical filtration $\KK$ of a simplicial complex, points of the persistence space $\mathrm{spc}({\V})$ are completely determined by the critical values of a discrete gradient vector field $\VV$ compatible with $\KK$. This claim is proven in Proposition \ref{prop:fiber}. The underlying idea to prove it is as follows.  

As $\mathrm{spc}({\V})$ can  be viewed as the fibered union of infinitely many persistence diagrams $\mathrm{dgm}({\V_L})$, with each $\V_L$ obtained by restricting $\V$ to a line $L$ with positive slope,   the filtration $\KK$ may also be restricted to $L$. This way   we obtain a 1-parameter filtration $\KK_L$, and $\V_L$ turns out to be the persistence module of $\KK_L$. Moreover, if $\KK$ has a compatible discrete gradient vector field $\VV$, then this discrete gradient vector field is inherited by $\KK_L$. Each critical cell  of $\VV$ has an entrance value in $\KK_L$ (as $L$ has positive slope). As is the case with $\KK$, the entrance values in $\KK_L$ of the critical cells in $\overline{C}$ identify elements in $\KK_L$ where the filtration may undergo a change in homotopy type, and therefore a change in homology.  Therefore, to determine $\mathrm{spc}({\V})$, it is enough to  identify the entrance values of critical cells of $\KK$ in the restricted filtration $\KK_L$. To this end, we introduce the following notation.

For every point  $u$ in $\RR^n$, let $S_+(u)$ be the positive cone with vertex $u$: 
$$S_+(u)=\{v\in\RR^n:u\preceq v\}$$ 
The boundary of the positive cone, $\partial S_+(u)$, decomposes into open faces.  In particular, $\partial S_+(u)$ can be partitioned by non-empty subsets $A$ of $[n]=\{1,2,\dots,n\}$ in the following way.  For $\varnothing\neq A\subseteq[n]$, define
\[
S_A(u)=\{(x_1,\dots,x_n)\in\RR^n| x_i=u_i \textrm { for }i\in A, x_j>u_j \textrm{ for }j\notin A\}.
\]
Then, for $A\neq B\subseteq[n]$, $S_A\cap S_B=\varnothing,$ and $\partial S_+(u)=\displaystyle\bigcup_{\varnothing\neq A\subseteq[n]}S_A.$

\begin{example}
If $n=2$ and $u=(u_1,u_2)$, the open faces of $\partial S_+(u)$ consist of the vertex $u$ and the two half-lines exiting from $u$ rightwards and upwards, respectively as shown in Figure \ref{fig:positive_cone}.

\[
\begin{split}
\partial S_+(u)&=S_{\{1,2\}}(u)\cup S_{\{1\}}(u)\cup S_{\{2\}}(u)\\
&=\{(u_1,u_2)\}\cup\{(x_1,x_2)\in\RR^2|x_1=u_1, x_2>u_2\}\\
&\quad \cup\{(x_1,x_2)\in\RR^2|x_1>u_1, x_2=u_2\}.
\end{split}
\]

\begin{figure}[H]
\begin{center}
\begin{tikzpicture}
\input{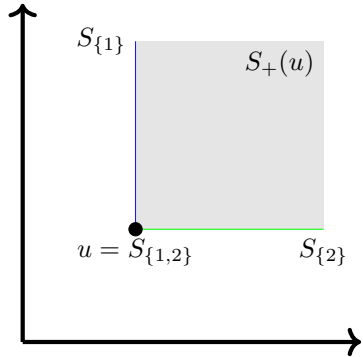}
\end{tikzpicture}
\caption{The positive cone $S_+(u)$ of $u\in\RR^2$ and the decomposition of its boundary into $S_{\{1\}}$, $S_{\{2\}}$ and $S_{\{1,2\}}$, which correspond respectively to the vertical boundary, horizontal boundary, and $u$.}\label{fig:positive_cone}
\end{center}
\end{figure}
\end{example}

It will be useful to consider the projection of points in the parameter space onto lines with positive slope (cf. \cite{Landi2018}).

\begin{definition}\label{def:push}
Given a line  $L\subseteq \RR^n$ with positive slope, for every $u\in\RR^n$ define
$$\mathrm{push}_L(u):=L\cap \partial S_+(u).$$ 
\end{definition}

\begin{figure}[H]
\begin{center}
\begin{tikzpicture}
\input{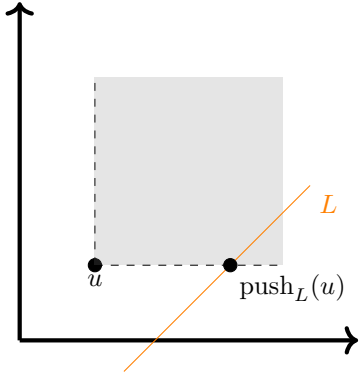}
\end{tikzpicture}
\caption{The push of $u$ along the line $L$. }\label{fig:push}
\end{center}
\end{figure}

\begin{proposition}[Properties of $\mathrm{push}_L(u)$]\label{prop:push-properties}
Some properties of $\mathrm{push}_L(u)$ are (see also Figure \ref{fig:push}): 

 \begin{enumerate}
 \item\label{push:1} $\mathrm{push}_L(u)$ consists of exactly one point because $L$ has positive slope. 
 \item\label{push:2} There is a unique non-empty subset $A^L_u$ of $[n]$  such that 
 $$ \mathrm{push}_L(u)=L\cap  S_{A^L_u}(u)$$
 For ease of notation, we concisely write $S_L(u)$ meaning $S_{A^L_u}(u)$. 
\item\label{push:3} $u\preceq \mathrm{push}_L(u)$ with equality only when $u\in L$.
\item\label{push:4} $ \mathrm{push}_L(u)$ is the smallest point on $L$ which is greater than or equal to $u$; smaller points on $L$ are either incomparable or less than $u$.
\item\label{push:6}  If $u\preceq v$, then $\mathrm{push}_L(u)\preceq\mathrm{push}_L(v)$.

\item\label{push:7} Let $u\preceq v$. Let $\varnothing\neq A,B\subseteq[n]$ such that $\mathrm{push}_L(u)\in S_A(u)$ and $\mathrm{push}_L(v)\in S_B(v)$.  We have:
\begin{enumerate}
\item[(a)] $S_A(u)\cap S_B(v)\ne \emptyset$ implies that $A\subseteq B$.

\item[(b)]  $\mathrm{push}_L(u)=\mathrm{push}_L(v)$ if and only if $S_A(u)\cap S_B(v)\ne \emptyset$.

\end{enumerate}
 \end{enumerate}
\end{proposition}
\begin{proof}
Properties \ref{push:1}, \ref{push:2}, \ref{push:3} and \ref{push:4} are immediate.  
\\

\noindent \textbf{Proof of property \ref{push:6}:} Suppose not; then $u\preceq v$ and $\mathrm{push}_L(v)\prec\mathrm{push}_L(u)$.  Note that, since both $\mathrm{push}_L(u)$ and $\mathrm{push}_L(v)$ are points on $L$ with positive slope, $\mathrm{push}_L(v)\prec\mathrm{push}_L(u)$ if and only if each coordinate is strictly less than, i.e., $(\mathrm{push}_L(v))_i<(\mathrm{push}_L(u))_i$ for all $i\in [n]$. Also, since $L$ has positive slope, there must exist at least one $j\in[n]$ such that $u_j=(\mathrm{push}_L(u))_j$.  Combining these, we have 
\[ 
v_j\leq(\mathrm{push}_L(v))_j<(\mathrm{push}_L(u))_j=u_j
\]
which contradicts $u\preceq v$.  So, the claim holds.
\\

\noindent \textbf{Proof of property \ref{push:7}(a):} Suppose there exists $y\in S_A(u)\cap S_B(v)$.  By definition, $y_i>v_i\geq u_i$ for all $i\notin B$; $y_i>u_i$ implies that $i\notin A$.  Thus, $i\notin B$ implies $i\notin A$, and the contrapositive must also be true, $j\in A$ implies $j\in B$.  

Note that the converse is not necessarily true; one could have $A=B=\{1\}$ but $u_1<v_1$, so that $S_A(u)\cap S_B(v)= \emptyset$.
\\

\noindent \textbf{Proof of property \ref{push:7}(b):} Indeed, $\mathrm{push}_L(u)=\mathrm{push}_L(v)$ implies that $S_A(u)\cap S_B(v)\neq\varnothing$.  
  Now, suppose  $S_A(u)\cap S_B(v)\neq\varnothing$.  By Property \ref{push:7}(a), this implies that $A\subseteq B$, and $u_i=v_i$ for $i\in A$.  Now, since $\mathrm{push}_L(u)$ and $\mathrm{push}_L(v)$ both belong to $L$, a line with positive slope, either $\mathrm{push}_L(u)=\mathrm{push}_L(v)$ or $\mathrm{push}_L(u)_j<\mathrm{push}_L(v)_j$ for all $j\in[n]$.  And, since $\mathrm{push}_L(u)\in S_A(u)$ and  $\mathrm{push}_L(v)\in S_A(v)$, $\mathrm{push}_L(u)_j=u_j=v_j=\mathrm{push}_L(v)_j$ for all $j\in A$.  Therefore, $\mathrm{push}_L(u)=\mathrm{push}_L(v)$.
  
\end{proof}

Recall the notations of Theorem \ref{thm:rankinv} where a single bar on some value $u\in\RR^n$, for which $\{u' \in \closure{C} | u'\preceq u\}$ is non-empty, denotes the greatest value in $\closure C$ less than or equal to that value:
\[\bar u:= \max \{u' \in \closure{C} | u'\preceq u\}.\]

 We also introduce a double bar notation that depends on a given line $L$ with positive slope (see also Figure \ref{fig:u_bar_bar}):
\[\bar{\bar u}^L:= \max \{u' \in \closure{C} |  u\preceq u'\textrm{ and } S_{L}(u)\cap S_{L}(u')\neq\varnothing \}.\]

\begin{figure}[H]
\begin{center}
\begin{tikzpicture}
\input{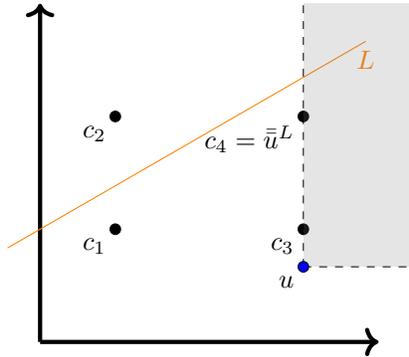}
\end{tikzpicture}
\caption{ The double bar of $u$ with respect to the line $L$ is the value $c_4 \in\overline{C}$. In this case we also have $\bar{\bar{c}}_3^L = \bar{\bar{u}}^L = c_4$. }\label{fig:u_bar_bar}
\end{center}
\end{figure}

\begin{lemma}\label{lem:double-bar}
For every $ u\in \closure{C}$,  it holds that 
$$\overline{\mathrm{push}_L( u)}=\bar{\bar u}^L.$$
\end{lemma}

\begin{proof}
First we note that for all $u\in\overline{C}$, using notation from Theorem \ref{thm:rankinv},

\[
\begin{split}
\overline{ \mathrm{push}_L( u)} &= \max \{u' \in \closure{C} | u'\preceq \mathrm{push}_L( u)\}\\
&=\max \{u' \in \closure{C} | u\preceq u'\preceq \mathrm{push}_L( u)\}\\
\end{split}
\]
as $u\in\overline{C}$, and $u\preceq\mathrm{push}_L(u)$.  Note that $u\preceq u'$ implies $ \mathrm{push}_L( u) \preceq\mathrm{push}_L( u')$, and $ u'\preceq \mathrm{push}_L( u)$ implies that $ \mathrm{push}_L(  u')\preceq  \mathrm{push}_L( \mathrm{push}_L( u))= \mathrm{push}_L( u)$.  So, we may write 

\[
\begin{split}
\overline{ \mathrm{push}_L( u)} &= \max \{u' \in \closure{C} |  u\preceq u' \textrm{ and }\mathrm{push}_L( u) =\mathrm{push}_L( u')\}.\\
\end{split}
\]

By definition of  $\varnothing\neq A^L_u,A^L_{u'}\subseteq[n]$ as the unique subsets $\varnothing\neq A^L_u,A^L_{u'}\subseteq[n]$ such that $\mathrm{push}_L(u)\in S_{A^L_u}(u)$ and $\mathrm{push}_L(u)\in S_{A^L_{u'}}(u')$.
Hence, $u\preceq u'$ and $ \mathrm{push}_L( u) =\mathrm{push}_L( u')$ if and only if $S_{A^L_u}(u)\cap S_{ A^L_{u'}}(u')\neq\varnothing $. So, finally we obtain, 

\[
\begin{split}
\overline{ \mathrm{push}_L( u)} &=\max \{u' \in \closure{C} |  u\preceq u'\textrm{ and } S_{A^L_u}(u)\cap S_{ A^L_{u'}}(u')\neq\varnothing  \}\\
&=\bar{\bar u}^L.
\end{split}
\]
that yields the claim recalling that the notation $S_L(u)$  is a shorthand for $S_{A^L_u}(u)$. 
\end{proof}

\begin{lemma}\label{lem:push-doublebar}
For all lines $L$ with positive slope, and for all $ u\prec  v\in \closure{C}$, we have 
$$\rho_\V(\mathrm{push}_L( u),\mathrm{push}_L( v))=\rho_\V(\bar{\bar u}^L,\bar{\bar v}^L).$$
\end{lemma}

\begin{proof}
Using Theorem \ref{thm:rankinv} and Lemma \ref{lem:double-bar}, we obtain
\[
\begin{split}
\rho_\V(\mathrm{push}_L( u),\mathrm{push}_L( v))&=\rho_\V(\overline{\mathrm{push}_L( u)},\overline{\mathrm{push}_L( v)})\\
&=\rho_\V(\bar{\bar u}^L,\bar{\bar v}^L)
\end{split}
\]
\end{proof}

\begin{lemma}\label{lem:bar=barbar}
For all lines $L$ with positive slope, and all $u\in L$, we have 
$$\bar u=\overline{\overline {(\bar u)}}^L.$$
\end{lemma}

\begin{proof}

By definition of double bar, $\bar u \preceq \overline{\overline{(\bar u)}}^L$ and $S_L(\bar u)\cap S_L(\overline{\overline {(\bar u)}}^L)\neq\varnothing$. So, by Proposition \ref{prop:push-properties}.\ref{push:7}(b),   
$$\mathrm{push}_L(\bar u)=\mathrm{push}_L\left(\overline{\overline{(\bar u)}}^L\right).$$ 
Additionally, Proposition \ref{prop:push-properties}.\ref{push:3} implies that $\overline{\overline{(\bar u)}}^L\preceq\mathrm{push}_L\left(\overline{\overline{(\bar u)}}^L\right)$, and, by Proposition \ref{prop:push-properties}.\ref{push:6}, we have  $\mathrm{push}_L(\bar u)\preceq \mathrm{push}_L(u)$ because $\bar u\preceq u$. 
Moreover, as $u\in L$, $\mathrm{push}_L(u)=u$. So finally we have
$$\bar u \preceq \overline{\overline{(\bar u)}}^L\preceq\mathrm{push}_L\left(\overline{\overline{(\bar u)}}^L\right)=\mathrm{push}_L(\bar u)\preceq \mathrm{push}_L(u)=u.$$
As $\overline{\overline{(\bar u)}}^L \in \overline{C}$,   the above ineqalities imply that $\bar u = \overline{\overline{(\bar u)}}^L$ by definition of $\bar u$.

\end{proof}

Our next goal is to prove that critical values of the discrete vector field on $K$ determine points of the persistence diagram of the restriction along a line through the parameter space.  Given such a line $L$, we may define   

\[\mathrm{push}_L( \closure{C})=\{\mathrm{push}_L(c)\ |\ c\in\closure{C}\}.
\]

Note that, since $\closure{C}$ is finite, so is $\mathrm{push}_L( \closure{C})$.  We can order the elements of $\mathrm{push}_L( \closure{C})$ as $c^1,c^2,\dots,c^m$, with $c^i\prec c^{i+1}$. 

%
%
%


\begin{proposition}\label{prop:fiber}
Let $L$ be a line  with positive slope. Let $\mathrm{push}_L( \closure{C})=\{c^1,c^2,\ldots, c^m\}$ be increasingly ordered. For all points $u\prec v$ on $L$, it holds that:
\begin{enumerate}
\item[(i)] If $u=c^i$ and $v=c^j$, then
\[
\mu_\V(u,v)=\rank_\V(c^i,c^{j-1})-\rank_\V(c^{i-1},c^{j-1})-\rank_\V(c^i,c^{j})+\rank_\V(c^{i-1},c^{j}),
\]
and $\mu_\V(u,v)=0$ if $u$ or $v$ not in $\closure{C}$.

\item[(ii)] If $u=c^i$, then 
\[
\mu_\V(u,\infty)=\rank_\V(c^i,c^{m})-\rank_\V(c^{i-1},c^{m}),
\]
and $\mu_\V(u,\infty)=0$ if $u$ not in $\closure{C}$.
\end{enumerate}
\end{proposition}

\begin{proof}
From $\closure{C}\ne \emptyset$ we get $\mathrm{push}_L( \closure{C})\ne \emptyset$. Note that we may partition the line $L$ by points of $\mathrm{push}_L( \closure{C})$. For each $c^i\in \mathrm{push}_L( \closure{C})$, we have 
$$c^i=\mathrm{push}_L( \overline{c^i}).$$
Indeed, by the bar notation, $d\preceq \overline{c^i}\preceq c^i$ for all $d\in \mathrm{push}_L^{-1}( {c^i})$. Thus, 
$$c^i=\mathrm{push}_L( d)\preceq \mathrm{push}_L( \overline{c^i})\preceq \mathrm{push}_L( {c^i})=c^i.$$

We first consider the case when $u\in L$ and $u\prec c^1$. In this case,  $\mu_\V(u,v)=0$ for all $v\succeq u$, and $\mu_\V(u,\infty)=0$. Indeed, we can take $\vec e\succ 0$ small enough so that, for all $0\prec \vec{e'}\preceq\vec e$, we have $u-\vec{e'}\prec u \prec u+ \vec{e'} \prec c^1$. Hence, Theorem \ref{thm:representative-dgm}  $\rho_\V(u-\vec{e'},v)=\rho_\V(u+\vec{e'},v)=0$ for all $v\succeq u$. 

We now consider the case of a point $u\in L$ such that $c^1\preceq u$. Let $c^i$ be the maximal element in $\mathrm{push}_L( \closure{C})$ such that $c^i\preceq u$. In this case we claim that
$$\overline{u}=\overline{c^i}$$
Indeed, suppose not.  Since $c^i\preceq u$, we must have $\overline{c^i}\preceq \overline{u}$.   Therefore,
\[
c^i=\mathrm{push}_L(\overline{c^i})\preceq \mathrm{push}_L(\overline{u})\preceq \mathrm{push}_L(u)=u.
\]
Since $\mathrm{push}_L(\overline{u})$ is in $\mathrm{push}_L( \closure{C})$ and $c^i$ is maximal in $\mathrm{push}_L( \closure{C})$ such that $c^i\preceq u$ (by assumption), we must have that $c^i= \mathrm{push}_L(\overline{u})$.  But then $\overline{u}\preceq c^i$,with $\overline{u}\in \closure{C}$,  implying  $\overline{u}\preceq \overline{c^i}$ by the bar notation.   So, it must be that $\overline{u}=\overline{c^i}.$
\\

Now, we are ready to prove the first statement in the case $c^1\preceq u$.  First, suppose that $u$ is not in $\mathrm{push}_L( \closure{C})$.  Then there is a maximal element in $\mathrm{push}_L( \closure{C})$ such that $c^i\prec u$.  If $i\neq m$, there exists $\vec e\succ 0$ such that 
\[
c^i\prec u-\vec e\prec u\prec u+\vec e\prec c^{i+1};
\]
if $i=m$, then there exists $\vec e\succ 0$ such that 
\[
c^m\prec u-\vec e\prec u\prec u+\vec e,
\]
and the above inequalities hold for all $0\prec \vec{e'}\preceq\vec e$.  Moreover, by the above claim, we have that for all such $\vec{e'}$, $\overline{u-\vec{e'}}=\overline{u+\vec{e'}}=\overline{c^i}$. 

Therefore,
\[
\begin{split}
\mu_{\V}(u,v)& = \min_{\scriptsize{\begin{array}{c}
\vec e\succ  0\\
u+\vec e\prec v-\vec e\end{array}}}\rank_{\V}(u+\vec e,v-\vec e)- \rank_{\V}(u-\vec e,v-\vec e)+\\ &\qquad\qquad\qquad-\rank_{\V}(u+\vec e,v+\vec e)+\rank_{\V}(u-\vec e,v+\vec e)\\
& = \min_{\scriptsize{\begin{array}{c}
\vec e\succ  0\\
u+\vec e\prec v-\vec e\end{array}}}\rank_{\V}(\overline{u+\vec e},v-\vec e)- \rank_{\V}(\overline{u-\vec e},v-\vec e)+\\ &\qquad\qquad\qquad-\rank_{\V}(\overline{u+\vec e},v+\vec e)+\rank_{\V}(\overline{u-\vec e},v+\vec e)\\
& = \min_{\scriptsize{\begin{array}{c}
\vec e\succ  0\\
u+\vec e\prec v-\vec e\end{array}}}\rank_{\V}(\overline{c^i},v-\vec e)- \rank_{\V}(\overline{c^i},v-\vec e)+\\ &\qquad\qquad\qquad-\rank_{\V}(\overline{c^i},v+\vec e)+\rank_{\V}(\overline{c^i},v+\vec e)\\
&=0.
\end{split}
\]
Similarly, if $v$ is not in $\mathrm{push}_L( \closure{C})$, we obtain $\mu_{\V}(u,v)=0.$

Now, if $c^j=v$ and $c^i=u$ for $c^i,c^j\in \mathrm{push}_L( \closure{C})$, then we can find $\vec e\succ0$ small enough such that both
\[
c^{i-1}\prec c^i-\vec e\prec c^i \prec c^i+\vec e\prec c^{i+1}
\]
and
\[
c^{j-1}\prec c^j-\vec e\prec c^j \prec c^j+\vec e\prec c^{j+1};
\]
(note that if $j=m$, then the second set of equalities does not have the final ``$\prec c^{j+1}"$ term.)

Additionally, for all $0\prec\vec{e'}\prec\vec e$:
\begin{itemize}
\item $\overline{c^j-\vec {e'}}=\overline{c^{j-1}}$ ,
\item $\overline{c^j+\vec {e'}}=\overline{c^j}$ ,
\item $\overline{c^i-\vec {e'}}=\overline{c^{i-1}}$ , and 
\item $\overline{c^i+\vec {e'}}=\overline{c^i}$.
\end{itemize}

Using these, we find that:
\[
\begin{split}
\mu_{\V}(u,v)& = \min_{\scriptsize{\begin{array}{c}
\vec e\succ  0\\
c^i+\vec e\prec c^j-\vec e\end{array}}}\rank_{\V}(c^i+\vec e,c^j-\vec e)- \rank_{\V}(c^i-\vec e,c^j-\vec e)+\\ &\qquad\qquad\qquad-\rank_{\V}(c^i+\vec e,c^j+\vec e)+\rank_{\V}(c^i-\vec e,c^j+\vec e)\\
& = \min_{\scriptsize{\begin{array}{c}
\vec e\succ  0\\
c^i+\vec e\prec c^j-\vec e\end{array}}}\rank_{\V}(\overline{c^i+\vec e},\overline{c^j-\vec e})- \rank_{\V}(\overline{c^i-\vec e},\overline{c^j-\vec e})+\\ &\qquad\qquad\qquad-\rank_{\V}(\overline{c^i+\vec e},\overline{c^j+\vec e})+\rank_{\V}(\overline{c^i-\vec e},\overline{c^j+\vec e})\\
& = \min_{\scriptsize{\begin{array}{c}
\vec e\succ  0\\
c^i+\vec e\prec c^j-\vec e\end{array}}}\rank_{\V}(\overline{c^i},\overline{c^{j-1}})- \rank_{\V}(\overline{c^{i-1}},\overline{c^{j-1}})-\rank_{\V}(\overline{c^i},\overline{c^j})+\rank_{\V}(\overline{c^{i-1}},\overline{c^j})\\
& =\rank_{\V}(\overline{c^i},\overline{c^{j-1}})- \rank_{\V}(\overline{c^{i-1}},\overline{c^{j-1}})-\rank_{\V}(\overline{c^i},\overline{c^j})+\rank_{\V}(\overline{c^{i-1}},\overline{c^j})\\
& =\rank_{\V}(c^i,c^{j-1})- \rank_{\V}(c^{i-1},c^{j-1})-\rank_{\V}(c^i,c^j)+\rank_{\V}(c^{i-1},c^j).\\
\end{split}
\]

To prove the second statement, we again first suppose that $u$ is not in $\mathrm{push}_L( \closure{C})$, and that $c^i$ is the maximal element in $\mathrm{push}_L( \closure{C})$ such that $c^i\prec u$.  Then, as in the proof of the first statement, we can find $0\prec \vec e$ such that, for all $0\prec \vec{e'}\preceq\vec e$, $\overline{u-\vec{e'}}=\overline{u+\vec{e'}}=\overline{c^i}$. Thus, 

\[
\begin{split}
\mu_{\V}(u,\infty) &= \min_{\scriptsize{\begin{array}{c}
\vec e\succ  0\\
 v\succeq u\end{array}}}\rank_{\V}(u+\vec e,v)- \rank_{\V}(u-\vec e,v)\\
 &= \min_{\scriptsize{\begin{array}{c}
\vec e\succ  0\\
 v\succeq u\end{array}}}\rank_{\V}(\overline{u+\vec e},v)- \rank_{\V}(\overline{u-\vec e},v)\\
 &= \min_{\scriptsize{\begin{array}{c}
\vec e\succ  0\\
 v\succeq u\end{array}}}\rank_{\V}(\overline{c^i},v)- \rank_{\V}(\overline{c^i},v)\\
 &=0.
\end{split}
\]

If $u=c^i$ for some $c^i\in\mathrm{push}_L( \closure{C})$, then we can find $0\prec \vec e$ such that, for all $0\prec \vec{e'}\preceq\vec e$, $\overline{u-\vec{e'}}=\overline{c^{i-1}}$ and $\overline{u+\vec{e'}}=\overline{c^i}$.  We also note that for all $v\in L$ such that $c^m\prec v$, $\overline{v}=\overline{c^m}$.  

Thus, 

\[
\begin{split}
\mu_{\V}(u,\infty) &= \min_{\scriptsize{\begin{array}{c}
\vec e\succ  0\\
 v\succeq u\end{array}}}\rank_{\V}(u+\vec e,v)- \rank_{\V}(u-\vec e,v)\\
 &= \min_{\scriptsize{\begin{array}{c}
\vec e\succ  0\\
 v\succeq u\end{array}}}\rank_{\V}(\overline{u+\vec e},v)- \rank_{\V}(\overline{u-\vec e},v)\\
  &= \min_{\scriptsize{\begin{array}{c}
 v\succeq u\end{array}}}\rank_{\V}(\overline{c^{i-1}},v)- \rank_{\V}(\overline{c^i},v)\\
   &= \min_{\scriptsize{\begin{array}{c}
 v\succeq u\end{array}}}\rank_{\V}(c^{i-1},v)- \rank_{\V}(c^i,v)\\
    &= \min_{\scriptsize{\begin{array}{c}
 v\succeq u\end{array}}}\rank_{\V}(c^{i-1},\overline{v})- \rank_{\V}(c^i,\overline{v}).\\
     &= \min_{\scriptsize{\begin{array}{c}
 v\succeq u\end{array}}}\rank_{\V}(c^{i-1},\overline{c^m})- \rank_{\V}(c^i,\overline{c^m})\\
      &= \rank_{\V}(c^{i-1},\overline{c^m})- \rank_{\V}(c^i,\overline{c^m})\\
     &= \rank_{\V}(c^{i-1},c^m)- \rank_{\V}(c^i,c^m).\\
\end{split}
\]
\end{proof}

\subsection{Grouping fibers of persistence spaces by equivalence}

We now use critical values to partition the set of all lines of $\RR^n$ into equivalence classes, as illustrated in Figure \ref{fig:line_eq}, such that the persistence diagrams of the restriction along lines in the same class are  easily obtainable from each other by a bijective correspondence.

\begin{definition}\label{def:rec-position} 
Two lines  $L,L'\subseteq \RR^n$ with positive slope are said to have the {\em same reciprocal position} with respect to $u$ if and only if $\mathrm{push}_L(u)$ and $\mathrm{push}_{L'}(u)$ belong to the same open face of $\partial S_+(u)$. Given a non-empty subset $U$ of $\RR^n$, we write $L\sim _{U}L'$, if $L$ and $L'$ have the same reciprocal position with respect to $u$ for all $u \in U$.  
\end{definition} 

\begin{example} Figure \ref{fig:line_eq} shows two examples of the equivalence classes of lines yielded by the set $\overline{C}$ of Example \ref{ex:bifiltration}.

\end{example}
\begin{figure} 
\begin{center}
\begin{tikzpicture}
\input{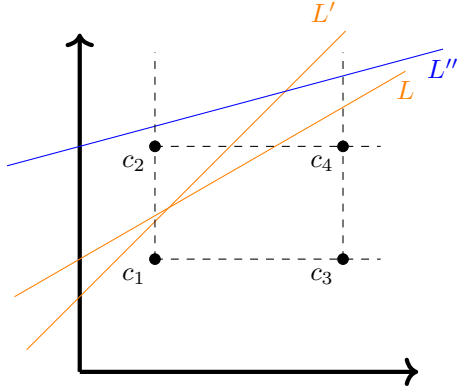}
\end{tikzpicture}
\caption{ The dashed lines represent the boundaries of the positive cones of the values in $\overline{C}$. Here $L$ and $L'$ have the same reciprocal position with respect to $\overline{C}$ but $L''$ does not. } \label{fig:line_eq}
\end{center}
\end{figure}

Lines with the same  reciprocal position with respect to $\closure{C}$ are characterized by the property of hitting the same face of the positive cone of $u$ for each $u\in \closure{C}$:

\begin{lemma}\label{lem:facets}
$L\sim _{\closure{C}}L'$ if and only if $S_{L}(u)=S_{{L'}}(u)$ for all values $u\in\closure C$. Hence, $A_u^L=A_u^{L'}$ for all values $u\in\closure C$.
\end{lemma}

\begin{proof}
Recall that $S_L(u)=S_{A_u^L}(u)$  and $A_u^L$ is the unique non-empty subset of $[n]$ such that $\mathrm{push}_L(u)=L\cap S_{A_u^L}(u)$.  Therefore, $\mathrm{push}_L(u)\in S_{A_u^L}(u)$ and $\mathrm{push}_{L'}(u)\in S_{A_u^{L'}}(u)$.

By definition, if $L\sim _{\closure{C}}L'$, then for all $u\in\closure C$, $\mathrm{push}_L(u)$ and $\mathrm{push}_{L'}(u)$ belong to the same open face of $\partial S_+(u)$, i.e. 

\[
S_{A_u^L}=S_{A_u^{L'}}
\]
 
 for all $u\in\closure C$. This can only happen if $A_u^L = A_u^{L'}$  for all $u\in\closure C$.
\end{proof}
 
 \begin{proposition}\label{prop:equiv-relation}
Given a non-empty subset $U$ in $\RR^2$, $L\sim_U L'$ defines an equivalence relation on the set of lines with positive slope. 
\end{proposition}

\begin{proof}
 We can define $\mathcal{A}_L(U)=\{A^L_u\}_{u\in U}$. By Lemma \ref{lem:facets}, $L\sim_U L'$ if and only if $\mathcal{A}_L(U)=\mathcal{A}_{L'}(U)$. Using this equivalent definition of $L\sim_U L'$, it is clear that $\sim_U$ is reflexive, transitive, and symmetric, and therefore an equivalence relation.
\end{proof}


The rank invariant on equivalent lines satisfies the following condition.

\begin{proposition}\label{prop:push-to-line}
If  $u\prec  v\in L$, with $\{u' \in \closure{C} | u'\preceq u\}$  non-empty, and  $L\sim _{\closure{C}}L'$, then it holds that 
$$\rank_\V(u,v)=\rank_\V(\mathrm{push}_{L'}( \bar u),\mathrm{push}_{L'}( \bar v)).$$
\end{proposition}

\begin{proof}
Since $u\preceq v$, it follows from Theorem \ref{thm:rankinv} that 
\[\rho_\V (u,v) = \rho_\V(\overline{u}, \overline{v}).\]
As $u,v\in L$,  $\bar u=\overline{\overline {(\bar u)}}^L$ and $\bar v=\overline{\overline {(\bar v)}}^L$ by Lemma \ref{lem:bar=barbar}, implying
\[ \rho_\V(\overline{u}, \overline{v})=\rank_\V\left(\overline{\overline{(\bar{u})}}^L, \overline{\overline{(\bar{v})}}^L\right).\]
As $\overline{u}, \overline{v}\in \closure{C}$, and  $u\preceq v$ implies $\overline{u} \preceq \overline{v}$,  from Lemma \ref{lem:push-doublebar} we get 
 \[  \rank_\V(\overline{\overline{(\bar{u})}}^L, \overline{\overline{(\bar{v})}}^L)=\rho_\V( \push_{L}(\overline{u}),\push_{L}(\overline{v})).
 \] 
\end{proof}

\begin{lemma}\label{lem:pushL=pushL'}
If $L\sim _{\closure{C}}L'$, then $\overline{\mathrm{push}_L( u)}=\overline{\mathrm{push}_{L'}( u)}$ for all  values $u\in\closure C$.
\end{lemma}

\begin{proof}
By Lemma \ref{lem:double-bar}, 
\[
\begin{split}
\overline{ \mathrm{push}_L( u)} &=\max \{u' \in \closure{C} |  u\preceq u'\textrm{ and } S_{L}(u)\cap S_{ L}(u')\neq\varnothing  \},\\
\overline{ \mathrm{push}_{L'}( u)}&=\max \{u' \in \closure{C} |  u\preceq u'\textrm{ and } S_{L'}(u)\cap S_{ L'}(u')\neq\varnothing  \}.
\end{split}
\]
So the claim follows  because  $S_{L}(u)=S_{L'}(u)$ and  $S_{L}(u')=S_{L'}(u')$ by  Lemma \ref{lem:facets}. 
\end{proof}

\begin{lemma}\label{lem:bar-push-bar}
If $L\sim _{\closure{C}}L'$, then $\overline{\mathrm{push}_{L'}( \bar u)}=\bar u$ for all  values $u\in L$.
\end{lemma}

\begin{proof}
It follows by successively applying Lemmas \ref{lem:pushL=pushL'}, and \ref{lem:double-bar}, \ref{lem:bar=barbar}.
\end{proof}

\begin{lemma}\label{lem:bijection}
If $L\sim _{\closure{C}}L'$, then the correspondence  $\sigma:\push_L(\closure C)\to \push_{L'}(\closure C)$  defined by $\sigma(d)=\push_{L'}(\push_L^{-1}(d))$ for all $d\in \push_L(\closure C)$,  is an order preserving  bijective function. In particular, $\sigma(d)=\push_{L'}(\bar d)$  with  $\bar d=\max\{u\in \closure{C}:u\preceq d\}$ as usual, for all $d\in \push_L(\closure C)$.
\end{lemma}

\begin{proof}
Let $d \in \push_L(\closure C)$. Then there exists at least one $c \in \closure C$ such that $d=\push_L(c).$  We first show that, for all $c\in \push_L^{-1}(d)$, 
\begin{eqnarray}\label{eq:sigma}
\push_{L'}(c)=\push_{L'}(\bar d) \ .
\end{eqnarray} 
By definition of double bar, we have that $S_L(\bar{\bar c}^L) \cap S_L(c) \neq \emptyset$. So, as $L \sim_{\closure C} L'$ implies that $S_L(\bar{\bar c}^L) = S_{L'}(\bar{\bar c}^L)$ and $S_L(c) = S_{L'}(c)$ by Lemma \ref{lem:facets}.  This means that 
$$S_{L'}(\bar{\bar c}^L) \cap S_{L'}(c) \neq \emptyset \ .$$
 Therefore, by Proposition \ref{prop:push-properties}.\ref{push:7}(b), we have that $$\push_{L'}(\bar{\bar c}^L)=\push_{L'}(c) \ .$$ Now note that $\bar d = \overline{\push_L(c)} = \bar{\bar c}^L$ by Lemma~\ref{lem:double-bar}, and so 
$$\push_{L'}(\bar d) = \push_{L'}(\bar{\bar c}^L) = \push_{L'}(c) \ .$$

Equality (\ref{eq:sigma}) implies that $\sigma$  is a well defined function because $\bar d$ is unique by 
Lemma~\ref{lem:max} and $\push_{L'}(\bar d)$ is also unique by Property \ref{push:1} of Proposition \ref{prop:push-properties}.

 Now for any $d' \in \push_{L'}(\closure C)$, there exists (at least one) $c' \in \closure C$ such that $d'=\push_{L'}(c')$ and $\sigma(\push_{L}(c')) = \push_{L'}(c')$, showing that $\sigma :\push_L(\closure C)\to \push_{L'}(\closure C)$ is surjective. 
 


We can analogously define a function $\tau: \push_{L'}(\closure C)\to \push_{L}(\closure C)$ by setting $\tau(d')=\push_L(\push_{L'}^{-1}(d'))$ for all $d'\in \push_{L'}(\closure C)$.

Now we prove that $\sigma$ and $\tau$ are bijective by showing that $\sigma$ is the inverse of $\tau$: for all $d'\in \push_{L'}(\closure C)$,
$$\sigma(\tau(d'))=\sigma(\push_L(\push_{L'}^{-1}(d')))=\push_{L'}(\push_{L'}^{-1}(d'))=d'$$
and, similarly, for all $d\in \push_{L}(\closure C)$,
$$\tau(\sigma(d))=\tau(\push_{L'}(\push_{L}^{-1}(d)))=\push_{L}(\push_{L}^{-1}(d))=d \ .$$


Finally, we show that $\sigma$ is order-preserving: Assume that $d,e \in \push_{L}(\closure C)$ with $d \preceq e$. Then $\bar d \in \closure C$ and  $\bar d \preceq d \preceq e$ by definition of bar, and therefore $$\bar d \preceq \bar e = \max \{u' \in \closure{C} | u'\preceq e\} \ .$$ Hence, by Proposition \ref{prop:push-properties}.\ref{push:6}, we have that $$\sigma(d) = \push_{L'}(\bar d) \preceq \push_{L'}(\bar e) = \sigma(e) \ ,$$ as required.  

\end{proof}

\begin{lemma}\label{lem:multi-injection}
Let  $L\sim _{\closure{C}}L'$ be equivalent lines with positive slope. For any $c^i\in \push_L(\closure C)=\{c^1,c^2,\ldots, c^m\}$ increasingly ordered,  let $d^i=\sigma(c^i)$ with  $\sigma$  as in Lemma \ref{lem:bijection}. Then, 
\[
  \mu_\V(d^i,d^j)=\mu_\V(c^i,c^j)
\]
for $c^i\prec c^j$, and 
\[
  \mu_\V(d^i,\infty)=\mu_\V(c^i,\infty).
\]

\end{lemma}

\begin{proof}
Assuming $c^i\prec c^j\in \push_L(\closure C)$,  by Proposition \ref{prop:fiber}(i),  \[
\mu_\V(c^i,c^j)=\rank_\V(c^i,c^{j-1})-\rank_\V(c^{i-1},c^{j-1})-\rank_\V(c^i,c^{j})+\rank_\V(c^{i-1},c^{j})
\]
 On the other hand, setting $d^h=\sigma(c^h)$ for  $h\in\{i,i-1\}$, and $d^k=\sigma(c^k)$ for  $k\in\{j,j-1\}$ yield
\[
\begin{split}
\rho_\V(c^h,c^k)&=\rho_\V(\overline{c^h},\overline{c^k}) \quad \mbox{(by Theorem \ref{thm:rankinv} )}\\
&=\rho_\V(\overline{\push_{L'}(\overline{c^h})},\overline{\push_{L'}(\overline{c^k})}) \quad \mbox{(by Lemma \ref{lem:bar-push-bar})}\\
&=\rho_\V(\push_{L'}(\overline{c^h}),\push_{L'}(\overline{c^k})) \quad \mbox{(by Theorem \ref{thm:rankinv})}\\
&=\rho_\V(d^h,d^k) \quad \mbox{(by definition).}\\
\end{split}
\]
Therefore,
$$\mu_\V(c^i,c^j)=\rank_\V(d^i,d^{j-1})-\rank_\V(d^{i-1},d^{j-1})-\rank_\V(d^i,d^{j})+\rank_\V(d^{i-1},d^{j})=\mu_\V(d^i,d^j)$$
with the second equality holding by Proposition \ref{prop:fiber}(i)  applied to $L'$.

Analogously,  by Proposition \ref{prop:fiber}(ii), we can  see that
$$\mu_\V(c^i,\infty)=\rank_\V(c^i,c^m)-\rank_\V(c^{i-1},c^m)=\rank_\V(d^i,d^m)-\rank_\V(d^{i-1},d^m)=\mu_\V(d^i,\infty).$$
\end{proof}

\begin{theorem}\label{thm:representative-dgm}
Let  $L\sim _{\closure{C}}L'$ be equivalent lines with positive slope,  parametrized by  $L: u=\vec m s+u_0$ and $L': u=\vec m' s'+u'_0$, respectively. Let  $\mathrm{dgm}(\V_L)$ and $\mathrm{dgm}(\V_{L'})$ be the persistence diagrams of the restrictions of $\V$ to $L$ and $L'$, respectively.
Then, there exists a multi-bijection (that is, a bijection between sets of points with multiplicities),  $\gamma: \mathrm{dgm}(\V_L)\to \mathrm{dgm}(\V_{L'})$ such that, for all $(s,t)\in \mathrm{dgm}(\V_{L})$,
$$ \gamma(s,t)=(s',t')\in \mathrm{dgm}(\V_{L'})$$ 
with
\begin{itemize}
\item $s'\in\RR$ such that $\vec m' s'+u'_0=\mathrm{push}_{L'}(\overline{\vec m s+u_0})$, and
\item $t'\in\RR$ such that $\vec m' t'+u'_0=\mathrm{push}_{L'}(\overline{\vec m t+u_0})$ if $t \in \RR$, while $t'=\infty$ if $t=\infty$.
\end{itemize}
\end{theorem}

\begin{proof}
For $(s,t)\in\mathrm{dgm}(\V_L)$ with $s<t<\infty$, we have $\mu_{\V_L}(s,t)=\mu_\V(u,v)$ with $u=\vec m s+u_0$ and $v=\vec m s+u_0$ in $L$. By Proposition \ref{prop:fiber}, we can see that $(u,v)$ will be of the form $(c^i,c^j)$ where $c^i\prec c^j\in\push_L(\closure C)$. 
By Lemma \ref{lem:multi-injection}, the bijection $\sigma: \push_L(\closure C)\to \push_{L'}(\closure C)$ such that $\sigma(c^i)=d^i$ with $d^i=\push_{L'}(\overline{c^i})$ for $1\le i\le m$, satisfies $\mu_\V(c^i,c^j)=\mu_\V(d^i,d^j)$ and $\mu_\V(c^i,\infty)=\mu_\V(d^i,\infty)$.
The parametrization of $L'$  uniquely determines values $s'<t'\in\RR$ such that $d^i=\vec m' s'+u'_0$ and $d^j=\vec m' t'+u'_0$. Since
$$\mu_{\V_{L'}}(s',t')=\mu_{\V}(d^i,d^j)= \mu_\V(c^i,c^j) = \mu_{\V_L}(s,t)$$
and
$$\mu_{\V_{L'}}(s',\infty)=\mu_{\V}(d^i,\infty)= \mu_\V(c^i,\infty) = \mu_{\V_L}(s,\infty)$$
we see that $\sigma$ induces the desired multi-bijection 
\[
\begin{split}
\gamma:\mathrm{dgm}(\V_L)&\to\mathrm{dgm}(\V_{L'})\\
\qquad (s,t)&\to (s',t')
\end{split}
\]
\end{proof}

\section{Conclusions and discussion}
\label{sec:conclusion}
Based on the results of this paper, we can derive a method to fiber the rank invariant of multi-parameter persistence along lines of positive slope chosen by a user. This way, the persistence space of an $n$-parameter persistence module is sliced into persistence diagrams. 

The method consists of an offline preprocessing step, where we compute the representative lines and their persistence diagrams, and an interactive step where we can compute the rank invariant for any chosen line in real time. 
Starting from a discrete gradient vector field $\VV$ consistent with the multi-filtration at hand as input data, the offline step requires:

\begin{itemize}
\item computing the set $ C$ of entrance values of the critical cells of  $\VV$,
\item taking the closure $\closure C$ of $C$ with respect to their least upper bound,
\item partitioning the set of lines with positive slope by the equivalence relation $\sim_{\closure C}$ and picking a representative line from each equivalence class (for example, following the procedure shown in Appendix~\ref{app:cutsandpairs}).
\item storing the persistence diagrams of the  restriction of filtration to each representative line.   
\end{itemize}

Having pre-computed these data, the interactive part 
\begin{itemize}
\item takes as input from the user a line $L$ with positive slope,
\item detects the representative line $L_0$ of its equivalence class with respect to $\sim_{\closure C}$, 
\item computes the persistence diagram relative to $L$ by pushing onto $L$ the bar of the persistence diagram relative to $L_0$. 
\end{itemize}
The correctness of the method is guaranteed by Theorem \ref{thm:representative-dgm}. 

The method requires additional routines from computational geometry in order to efficiently detect representatives of equivalence classes of lines, computing the bars of points, and pushing points onto lines. In Appendix \ref{app:cutsandpairs}, we propose a method applied to $2$-parameter persistence modules to find representative lines for the equivalence classes defined in Section \ref{sec:computing_pers_space}. The method is based on a bijection between segments linking points of $\overline{C}$ and lines cutting $\overline{C}$ in two non-empty subsets.

Moreover, the method requires routines for the persistence diagram computation such as those implemented in \cite{GUDHI}, \cite{PHAT}, or \cite{TTK}. It is worth noticing that, based on the tests using \cite{PHAT} presented in \cite{Scaramuccia-et-al-2020},
it is more efficient to compute persistence diagrams of the persistence module restricted to lines starting with the Morse complex obtained from $\VV$ rather than directly from the original cell complex. In other words, the gradient vector field $\VV$ is used twice: its critical cells allow  us both to determine representative lines, and  to reduce persistence computation

In conclusion, the presented  method allows for:  computational efficiency,  by requiring only linear asymptotic time complexity to obtain the input gradient field, e.g. with the algorithm of \cite{Scaramuccia-et-al-2020}; theoretical improvements, by permitting any number of parameters; data analysis and understanding advantages, by making explicit the correspondence between persistence features and critical cells. For future work we plan to extend the algorithm to a larger number of parameters. 

 The rank invariant fibering along lines is central also in the definition of the {\em matching distance}, a metric on rank invariants of multi-parameter persistence modules \cite{Cerri2016, Cerri-et-al2013} ensuring their stability.   
  In \cite{Kerber-Lesnick-Oudot2018}, the exact computation of the matching distance is achieved for at most two parameters in polynomial runtime in the number of simplices requiring $O( m^{11})$ runtime and $O( m^4 )$ memory, with $m$ denoting the number of simplices.  
  Motivated by the practical need of decreasing the number of operations, and increasing the number of allowed parameters,  our next project will be to  extend equivalence classes of lines to pairs of persistence modules and to apply the method presented here to the matching distance exact computation in any number of parameters.

\section*{Acknowledgements}
This research began at the 2019 Women in Computational Topology (WinCompTop) workshop in Canberra. We thank Ashleigh Thomas and Elizabeth Stephenson for joining us in the initial discussions during that week. The results of this paper have been presented within the Summer 2020 AATRN Seminars.
We thank Anand Deopurkar, Anthony Licata, and Nicholas Proudfoot for helpful conversations related to Appendix~\ref{app:cutsandpairs}.

\bibliography{rank_invariant}
\bibliographystyle{plain}  

\renewcommand{\appendixpage}{{\raggedright \textbf{\huge Appendix}}}

\appendixpage
\begin{appendices}
\section{Enumerating equivalence classes of lines}\label{app:cutsandpairs}
In this appendix we describe an algorithm to enumerate the equivalence classes of lines with positive slope with respect to a set $\closure{C}$ in $\RR^2$. The results of this section will be initially presented in a slightly more general form.

Consider a set \(P\) of \(n\) points in the real plane, not necessarily in general position.
We say that  any division of \(P\) into two non-empty subsets via a line not passing through any point of \(P\) is a \emph{cut} in the plane.
Note that a cut determines an equivalence class of lines if and only if the cut can be realised by a line of positive slope.

\subsection{Cuts are determined by primitive pairs}
A key observation is that we can associate a certain pair of points in \(P\) to any given cut.
We say that a pair of distinct points in \(P\) is \emph{primitive} if no other point of \(P\) lies in the interior of the segment joining the two points.
\begin{proposition}
  There is a bijection between the primitive pairs of points in \(P\) and the cuts of \(P\).
\end{proposition}

\begin{proof}
  We give an explicit bijection, as follows: given a primitive pair of points in \(P\), we can rotate the line joining the two points clockwise by a small amount, around the midpoint of the segment.
  It is possible to choose a small enough angle so that this rotated line does not pass through any point of \(P\).
  Moreover, the two original points lie on opposite sides of this line, and so this line defines a cut.

  Let us now prove that to every cut we can associate a pair of points as above.
  Let \(\{A,B\}\) be a cut, where \(A\) and \(B\) are non-empty disjoint subsets of \(P\) with \(P = A \cup B\).
  We can represent this cut by a dividing line \(L\).

Now fix a direction \(\vec{v}\) orthogonal to \(L\).
    Translate \(L\) along \(\vec{v}\) until it hits at least one point of \(P\) for the first time.
    Let \(H\) be the set of points that \(L\) hits.
    If \(H\) has only one element, set \(a\) to be the unique element of \(H\). If \(H\) has more than one element, then set \(a\) to be the one with the following property: for any other point \(a_1 \in H\), the dot product of the vector \(\vec{v}\) with the vector \(\overrightarrow{aa_1}\) is positive. In other words, if we consider \(\vec{v}\) to be the positive direction along the \(x\)-axis, then \(a\) is the point with the minimum \(y\)-coordinate (see Figure \ref{fig:center_rotation}).

\begin{figure}[t]
\begin{center}
\begin{tikzpicture}
\input{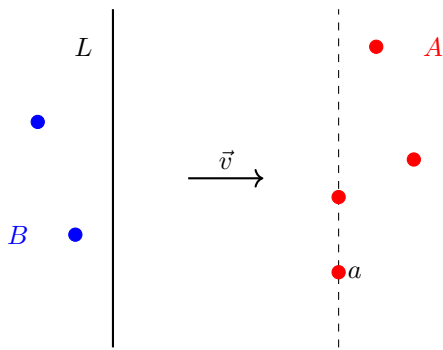}
\end{tikzpicture}
\caption{Choosing the initial center of rotation $a$} \label{fig:center_rotation}
\end{center}
\end{figure}
  Note that until the line hits \(a\), it still defines the same cut.
  Without loss of generality, assume that \(a \in A\).
  Taking \(a\) to be the centre of rotation, rotate the line \(L\) counter-clockwise until it hits another point of \(P\).  Note that the line defines a cut of all points of \(P\) other than those on this line, and it is the same cut except for these points.

 If the line simultaneously hits multiple points, then exactly one of the following is true of these points: they all points lie in \(A\), they all lie in \(B\), or they include both points in \(A\) and in \(B\). In the third case, \(a\) lies between the rest of the points in \(A\) and those in \(B\) on the line. 
  Now we have an algorithm to generate the primitive pair $a,b$ for the cut $L$, as follows.  
  \begin{enumerate}
  \item If at least one point hit lies in \(B\), set \(b\) to be the point of \(B\) hit by the line that is closest to \(a\), as shown in Figure \ref{fig:step1-1b} (Left).
 
    In this case, rotating the line \(ab\) clockwise around the midpoint of the segment \(ab\) gets us back the original cut.
    This is because the line yields the original cut excluding the points of \(P\) that lie on it, and therefore rotating it slightly clockwise restores these points to the correct subsets (either \(A\) or \(B\)).
    We are done.
  \item If all points hit lie in \(A\), reset the center of rotation to be the point of \(A\) hit by the line that is farthest away from \(a\) as in Figure \ref{fig:step1-1b} (Right). Rename this point as \(a\). Continue to rotate the line counter-clockwise around the new centre of rotation until it hits another point of \(P\), and repeat the above steps until the algorithm terminates.
  \end{enumerate}

 \begin{figure}[t]
\begin{center}
\begin{tikzpicture}
\input{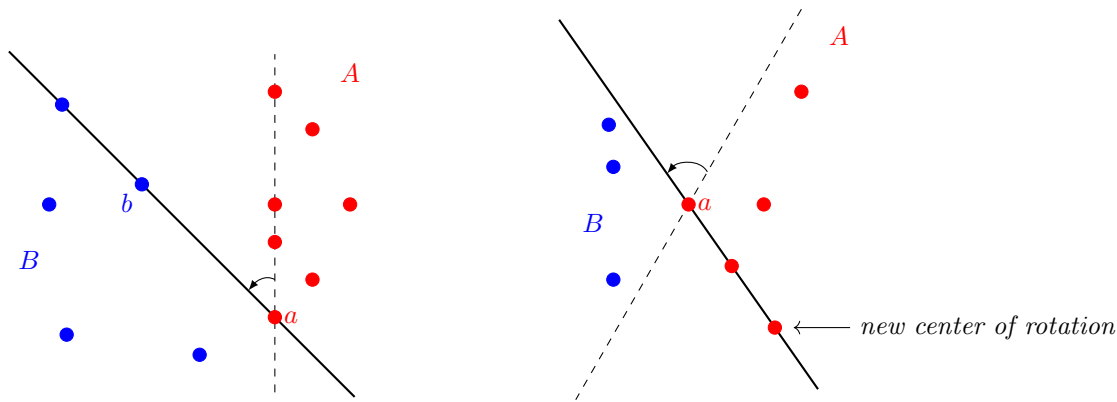} 
\end{tikzpicture}
\caption{Left: Rotating $L$ around $a$ will hit at least one point in $B$. Right: Rotating $L$ around $a$ hits only  points  in $A$.} \label{fig:step1-1b} 
\end{center}
\end{figure}

  The algorithm terminates when the rotating line hits a point of \(B\).
  This always happens, for the following reason.
  As the line makes a full rotation around the convex hull of \(A\), it sweeps through the entire plane except for the convex hull of \(A\), while \(B\) is a non-empty set outside the convex hull of \(A\).
\end{proof}

\begin{remark}\label{rem:ccw-algorithm}
  There is a similar, counter-clockwise, bijection between cuts and primitive pairs: rotate a line segment joining a primitive pair by a small amount counter-clockwise.
  Thus each cut corresponds to a ``clockwise primitive pair'', and another ``counter-clockwise primitive pair''.
  These pairs are distinct unless all points of \(P\) lie on a single line.
  Correspondingly, there is a counter-clockwise version of the algorithm explained in the previous proposition (see Figure \ref{fig:primitive}).
  We use this fact in the next section.
\end{remark}

\subsection{Achieving positive slope}\label{subsect:positiveslope}
We have determined that every primitive pair of points in \(P\) determines a cut by rotating the line segment through this pair slightly clockwise about the midpoint of the segment.
This is the \emph{clockwise primitive pair} associated to this cut.
Similarly, every primitive pair of points determines another cut by rotating the line segment through this pair slightly counter-clockwise around the midpoint of the segment.
This is the \emph{counter-clockwise primitive pair} associated to this cut.
These are shown in Figure~\ref{fig:primitive}.
\begin{figure}[h]
\begin{center}
\begin{tikzpicture}
\input{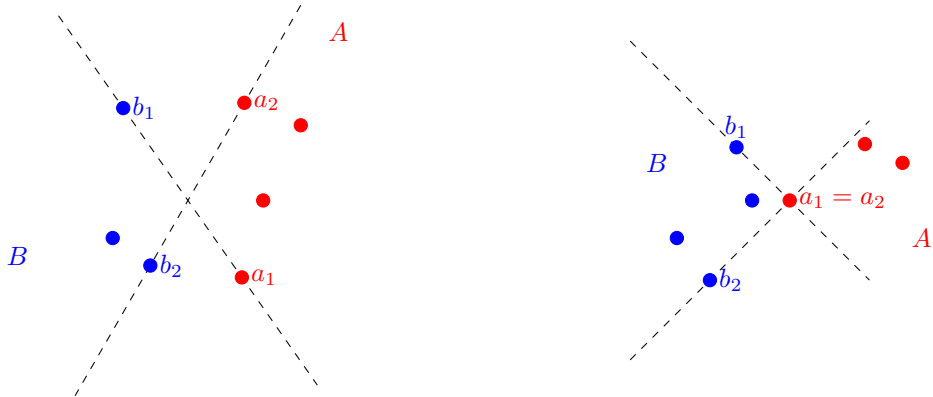}
\end{tikzpicture}
\caption{The clockwise and counter-clockwise primitive pairs associated to the  cut $\{A,B\}$. The pair $(a_1,b_1)$ is the clockwise primitive pair and $(a_2,b_2)$ is the counter-clockwise primitive pair. } \label{fig:primitive}
\end{center}
\end{figure}
The clockwise and counter-clockwise primitive pairs associated to any cut can be found by the algorithm in the previous section, and its variant explained in Remark~\ref{rem:ccw-algorithm} respectively.
We now tackle the problem of determining whether such a cut can be defined by a line of positive slope.
The answer is given by Algorithm~\ref{alg:pos-slope}.
\begin{algorithm}[h]
  \caption{An algorithm to determine whether a cut can be obtained by a line of positive slope.}\label{alg:pos-slope}
  \begin{algorithmic}[1]
    \STATE Let \((a_1,b_1)\) and \((a_2,b_2)\) be the clockwise and anti-clockwise primitive pairs respectively for a cut \(\{A,B\}\).
    \STATE Let \(m_1\) and \(m_2\) be the slopes of the lines \(a_1b_1\) and \(a_2b_2\) respectively.
    \IF{\(a_1 = b_1\) and \(a_2 = b_2\)}
    \RETURN \TRUE
    \ENDIF
    \IF{\(0 < m_1 \leq \infty\)}
    \RETURN \TRUE
    \ELSIF{\(0 \leq m_2 < \infty\)}
    \RETURN \TRUE
    \ELSE
    \STATE We have $-\infty<m_1\leq 0$ and $-\infty\leq m_2<0$.
    \IF{$m_1<m_2$}
    \RETURN \TRUE
    \ELSE
    \RETURN \FALSE
    \ENDIF
    \ENDIF
  \end{algorithmic}
\end{algorithm}

We will need some preparatory lemmas to prove the correctness of the algorithm.
We use the same setup for all of these lemmas, as follows.
Let \((a_1,b_1)\) and \((a_2,b_2)\) be the clockwise and anti-clockwise primitive pairs respectively for a cut \(\{A,B\}\).
It may be the case that either \(a_1 = a_2\) or \(b_1 = b_2\), but we suppose that not both are true.
This supposition is true unless all points in $P$ lie on a line.

Let \(p\) be the intersection point of the segments \(a_1b_1\) and \(a_2b_2\).
Note that \(p\) must exist, for the following reason.
If all four points are distinct, then \(a_2\) and \(b_2\) lie on opposite sides of the line \(a_1b_1\), and so the segments intersect somewhere in their interiors.
Otherwise, if \(a_1 = a_2\) (resp. \(b_1 = b_2\)), then \(p = a_1 = a_2\) (resp. \(p = b_1 = b_2\)).
\begin{lemma}\label{lem:line-sweep-quad}
  Suppose that the four points \(a_1, a_2, b_1, b_2\) are distinct.
  If we start at the line \(a_1b_1\) and rotate clockwise around \(p\) until we hit \(a_2b_2\), every intermediate line defines the cut \(\{A,B\}\).
\end{lemma}
\begin{proof}
  Let \(L_1\) be the line \(a_1b_1\) and \(L_2\) be the line \(a_2b_2\).
  Note that each of the two lines has a well-defined ``$A$'' side and a well-defined ``$B$'' side: the side of the line on which the remaining points of \(A\) lie is the ``$A$'' side, and the side on which the remaining points of \(B\) lie is the ``$B$'' side.

  The lines \(L_1\) and \(L_2\) cut up the plane into four open cones.
  We can label these cones as \(C_{AA}, C_{AB}, C_{BB}, C_{BA}\), where for example \(C_{BA}\) is the intersection of the ``$B$'' side of $L_1$ with the ``$A$'' side of $L_2$.
  The cones \(C_{AB}\) and \(C_{BA}\) contain no points of \(P\).
  This is precisely because these cones lie on the ``$A$'' side of one of the lines and on the ``$B$'' side of the other.
  Moreover, their closures only intersect at the point \(p\).

  It is clear that any line rotated clockwise around \(p\) starting from \(L_1\) until we hit \(L_2\), excluding \(L_1\) and \(L_2\) itself, lies completely in the set \(C_{AB} \cup C_{BA} \cup \{p\}\) and has $a_1$ on its ``$A$'' side and $b_1$ on its ``$B$'' side.  So any such line continues to define the same cut \(\{A,B\}\).
\end{proof}

\begin{lemma}\label{lem:line-sweep-tri}
  Suppose that \(a_1=a_2\). 
  If we start at the line \(a_1b_1\) and rotate clockwise around \(a_1=a_2\) until we hit \(a_1b_2\) keeping track of the trace of $b_1$ under this rotation, then every intermediate line, except for \(a_1b_2\) itself, defines the cut \(\{A,B\}\) after a sufficiently small clockwise rotation about the midpoint between $a_1$ and the trace of $b_1$. 
\end{lemma}
\begin{proof}
As in the proof of Lemma~\ref{lem:line-sweep-quad}, let \(L_1\) be the line \(a_1b_1\) and \(L_2\) be the line \(a_1b_2\), and notice that $L_1$ and $L_2$ cut up the plane into four open cones \(C_{AA}, C_{AB}, C_{BB}, C_{BA}\). As in the proof of Lemma~\ref{lem:line-sweep-quad}, the cones  \(C_{AB}\) and \(C_{BA}\) contain no points of $P$ and their closures only intersect in the point of intersection of $L_1$ and $L_2$, which is $a_1$ in this case. 

It is clear that any line rotated clockwise around \(a_1\) starting from \(L_1\) until we hit \(L_2\), excluding \(L_1\) and \(L_2\) itself, lies completely in the set \(C_{AB} \cup C_{BA} \cup \{a_1\}\) and has $b_1$ on its ``$B$'' side, so defines the cut \(\{A,B\}\) excluding $a_1$. A sufficiently small clockwise rotation of such a line about the midpoint between $a_1$ and the trace of $b_1$ moves $a$ to its ``$A$'' side without crossing any other points in $P$, so the resulting line indeed defines the cut \(\{A,B\}\).
 
\end{proof}

\begin{lemma}\label{lem:slope-sweep}
  Let \(S\) be the set of possible slopes of lines obtained by starting at the line \(a_1b_1\) and rotating clockwise through \(p\) until we hit the line \(a_2b_2\), excluding the slopes of the lines \(a_1b_1\) and \(a_2b_2\) themselves.
  Then the slope of any line \(L\) that defines the same cut \(\{A,B\}\) lies in \(S\).
\end{lemma}
\begin{proof}
  For the proof of this lemma, the four points \(a_1,a_2,b_1,b_2\) need not all be distinct.
  First note that if \(L\) is any line defining the cut \(\{A,B\}\), then it intersects the interiors of the segments \(a_1b_2\) and \(a_2b_1\).
  This is precisely because \(\{a_1,a_2\}\) and \(\{b_1,b_2\}\) lie on opposite sides of \(L\).

  Now suppose \(L\) is \emph{any} line that intersects the interiors of the segments \(a_1b_2\) and \(a_2b_1\).
  These are two opposite sides of the (possibly degenerate) quadrilateral \(a_1a_2b_1b_2\).
  Therefore \(L\) must also intersect the interiors of the diagonals of this (possibly degenerate) quadrilateral, namely the segments \(a_1b_1\) and \(a_2b_2\).
  In particular, because \(L\) intersects both \(a_1b_1\) and \(a_2b_2\), it cannot have slope equal to either \(m_1\) or \(m_2\).
  
  The set of possible slopes of lines in the plane can be identified with the real projective line, by noting that slopes can lie between \([-\infty, \infty]\) with \(\infty = -\infty\).
  We now have a continuous map
  \begin{equation}\label{eq:segment-map}
    a_1b_2 \times a_2b_1 \to \mathbb{RP}^1,
  \end{equation}
  defined by mapping an ordered pair of points to the slope of the line joining the two points.
  By the previous argument, the image of this map lies in \(\mathbb{RP}^1\setminus \{m_1,m_2\}\), which has two connected components.
  Since the domain is connected, the image of the map must lie in exactly one of the connected components.

  The set \(S\) is precisely one of the two connected components: we start at \(m_1\), rotate clockwise until we hit \(m_2\).
  The other connected component is obtained by rotating counter-clockwise starting at \(m_1\) until we hit \(m_2\).

  We already know by either Lemma~\ref{lem:line-sweep-quad} or Lemma~\ref{lem:line-sweep-tri} (depending on whether or not the four points \(a_1, a_2,b_1,b_2\) are distinct) that there are points in the image of the map in Equation~\ref{eq:segment-map} that lie in \(S\).
  By connectedness, all lines that intersect the interiors of \(a_1b_2\) and \(a_2b_1\) have slopes that lie in \(S\). In particular, all lines that define the same cut \(\{A,B\}\) have slopes that lie in \(S\).
\end{proof}

Now we can prove the correctness of the algorithm.
\begin{proposition}\label{prop:alg-correct}
  Algorithm~\ref{alg:pos-slope} correctly determines whether a cut can be obtained by a line of positive slope.
\end{proposition}
\begin{proof}
  Recall that \((a_1,b_1)\) and \((a_2,b_2)\) are the clockwise and anti-clockwise primitive pairs respectively for a cut \(\{A,B\}\).
  Recall that \(m_1\) and \(m_2\) are the slopes of the lines \(a_1b_1\) and \(a_2b_2\) respectively.
  We treat each step of the algorithm in order.

  First, \(a_1 = a_2\) and \(b_1 = b_2\) if and only if all points of \(P\) lie on a single line.
  In this case it is clearly always possible to achieve any cut by a line of positive slope.
  Now assume that not all points of \(P\) lie on a single line, which implies that either \(a_1 \neq a_2\) or \(b_1 \neq b_2\).
  This is the setting of the previous lemmas.
  
  If \(0 < m_1\leq \infty\), then a small clockwise rotation of the line \(a_1b_1\) has positive slope.
  Since the rotated line determines the desired cut \(\{A,B\}\), we are done.
  Similarly, if \(0 \leq m_2 < \infty \), then a small counter-clockwise rotation of the line \(a_2b_2\) has positive slope.
  Since the rotated line determines the desired cut \(\{A,B\}\), we are done.

  Now suppose that \(-\infty < m_1 \leq 0\) and \(-\infty \leq m_2 < 0\).
  Let \(p\) be the intersection point of \(a_1b_1\) and \(a_2b_2\).
  Let \(S\) be the set of possible slopes of lines obtained by starting at the line \(a_1b_1\) and rotating clockwise through \(p\) until we hit the line \(a_2b_2\).
  
  Suppose first that \(-\infty < m_1 < m_2 < 0\).
  As we sweep clockwise from \(a_1b_1\), we begin at \(m_1\), decrease slope until we hit a vertical line with slope \(-\infty = \infty\), and then decrease again from \(\infty\) until we cross \(0\) down to \(m_2\).
  In particular, at least one of the intermediate lines has positive slope.
  If the points \(a_1,a_2,b_1,b_2\) are all distinct, then by Lemma~\ref{lem:line-sweep-quad} we have a line of positive slope that gives the cut \(\{A,B\}\).
  If two of the four points are equal, then Lemma~\ref{lem:line-sweep-tri} states that a sufficiently small clockwise rotation of one of the intermediate lines (which will also have positive slope) gives the cut \(\{A,B\}\).

  Now suppose that \(m_1 \geq m_2\).
  In this case, the set \(S\) consists only of negative numbers: these are the slopes starting from \(m_1\) and decreasing down to \(m_2\).
  By Lemma~\ref{lem:slope-sweep}, we see that there is no line of positive slope that defines this cut.
\end{proof}

\subsection{Cuts through a fixed point}\label{subsect:c-cuts}
In order to address equivalence classes of lines that pass through a given point $c$ of $\closure C$, we now  say that a division of a non-empty set of points \(P\)   of the plane, with $c\notin P$, into  two disjoint subsets $A$ and $B$, of which at most one can be empty, via a line  passing through the given point $c$ and disjoint from \(P\), is a $c$-\emph{cut} of $P$. Reciprocally, we say that a line through $c$ and a point of \(P\) is a $c$-\emph{primitive line} of $P$.

\begin{proposition}
  There is a bijection between the $c$-primitive lines of  \(P\) and the $c$-cuts of \(P\).
\end{proposition}

\begin{proof}
 We construct an explicit bijection, called the \emph{clockwise bijection}, as follows.  Given a $c$-primitive line of \(P\), we can rotate this line clockwise by a small amount, around $c$, so that this rotated line does not pass through any point of \(P\). This line defines a $c$-cut.
Vice versa, with every $c$-cut of $P$ we can associate a $c$-primitive line by rotating the line realing the $c$-cut counter-clockwise until it hits some point of $P$, which exists because $P$ is non-empty. Note that by a completely symmetric argument we also have a \emph{counter-clockwise bijection}.
\end{proof}

We have determined that any $c$-cut $L$ of $P$ can be determined by rotating  both  a $c$-primitive line  $L_1$ clockwise and a $c$-primitive line  $L_2$ counter-clockwise. Let $m_1$ and $m_2$ be the slopes of $L_1$ and $L_2$, respectively. In the case $m_1=m_2$, because $L_1$ and $L_2$ both pass though $c$, we have $L_1=L_2$. In this case all points of $P$ belong to $L_1$ and positive slope can be achieved by rotation around $c$ for every value of $m_1$.

In the case when $m_1\ne m_2$, Algorithm \ref{alg:pos-slope} applied to $c$-primitive lines instead of primitive pairs of points achieves the goal. Indeed, again, $L_1$ and $L_2$ have a well defined $A$-side and $B$-side, and $L-\{c\}$ is contained in $C_{AB}\cup C_{BA}$. Then the argument follows as in Proposition \ref{prop:alg-correct}.

\subsection{Retrieving representatives lines}

With reference to the equivalence relation on lines defined by their reciprocal position with respect to the set $\closure C$ of critical values and their least upper bounds  as given in Definition \ref{def:rec-position}, our goal is to retrieve a representative line with positive slope for each possible equivalence class.  Recall from Lemma \ref{lem:facets} that two lines belong to the same equivalence class with respect to $\closure{C}$ if and only if they hit the positive cone of each point of $\closure{C}$ at the same facet. 


There are three possible situations for lines in the same class: (i) There is only one line in the equivalence class passing through two points $c$ and $c'$ of $\closure{C}$; (ii)  the lines in the considered equivalence class  contain exactly one point of $\closure{C}$, say $c$; 
(iii) the lines in the considered equivalence class do not contain any point of $\closure{C}$. 

Case (i) can be easily solved by taking lines through all possible pairs of distinct  points $c$ and $c'$ in $\closure{C}$, provided that $c\preceq c'$, paying attention to not taking the same line multiple times if there are  more than two points on the same line.

Case (ii). In this case, each such line partitions $\closure{C}-\{c\}$ into two subsets $A$ and $B$. For each equivalence class of lines for which $A$ and $B$ are both empty, $c$ is the only point of $\closure c$, so there is only one such equivalence class and any line through $c$ with positive slope is a representative of it. In the case when at least one between $A$ and $B$ are non-empty,  we can obtain a representative line by applying the algorithm presented in Subsection \ref{subsect:positiveslope} as explained in Subsection \ref{subsect:c-cuts}.  Note that since $\closure{C}$ is closed under least upper bound, the case where $-\infty<m_1<m_2<0$ (line 12 in Algorithm \ref{alg:pos-slope}) cannot occur.

Case (iii). In this case, each such line partitions $\closure{C}$ into two subsets $A$ and $B$. For each equivalence class of lines for which $A$ and $B$ are both non-empty, we can obtain a representative by  applying the algorithm presented in Subsection \ref{subsect:positiveslope}. For the case when either $A$ or $B$ is empty, the other one is necessarily equal to $\closure{C}$. There are exactly two such equivalence classes of lines depending on whether the lines hit all the positive cones of points of $\closure{C}$ at their horizontal or vertical facets. As a representative of the first class, we can take a line parallel to the diagonal of $\RR^2$ passing to a point with abscissa greater  than the maximum abscissa of points of $\closure{C}$, and ordinate smaller than the minimum ordinate of points of $\closure{C}$. Symmetrically, as a representative of the second class, we can take a line parallel to the diagonal of $\RR^2$ passing to a point with abscissa smaller  than the minimum abscissa of points of $\closure{C}$, and ordinate greater than the maximum ordinate of points of $\closure{C}$.

\end{appendices}

\end{document}